\documentclass[a4paper,10pt,leqno]{amsart}
\usepackage{amssymb, amsmath}
\newtheorem{thm}{Theorem}[section]
\newtheorem{lem}[thm]{Lemma}

\theoremstyle{definition}
\newtheorem{rmk}{Remark}

\allowdisplaybreaks
\usepackage{color}

\DeclareMathOperator{\supp}{supp}
\DeclareMathOperator{\dd}{div}
\newcommand{\pp}{\partial}
\newcommand{\va}{\varphi}
\newcommand{\la}{\lambda}
\newcommand{\low}{I_{\text{low}}}

\title[
Inverse Problems for Fractional Diffusion Equations
]{
Inverse problems for a half-order time-fractional diffusion equation in arbitrary dimension by Carleman estimates
}

\author[X. Huang]{Xinchi Huang}
\address{%
Department of Mathematical Sciences, 
The University of Tokyo, 
Komaba, Meguro, Tokyo, 153-8914, Japan%
}%
\email{huangxc@ms.u-tokyo.ac.jp}
\author[A. Kawamoto]{Atsushi Kawamoto}
\address{%
Department of Mathematical Sciences, 
The University of Tokyo, 
Komaba, Meguro, Tokyo, 153-8914, Japan%
}%
\email{kawamo@ms.u-tokyo.ac.jp%
}

\date{}

\begin{document}
%
%
\begin{abstract}
We consider a half-order time-fractional diffusion equation in an arbitrary dimension 
and investigate inverse problems of determining the source term or the diffusion coefficient 
from spatial data at an arbitrarily fixed time under some additional assumptions.   
We establish the stability estimate of Lipschitz type in the inverse problems and 
the proofs are based on the Bukhgeim-Klibanov method by using Carleman estimates. 
\end{abstract}

\maketitle


%
%
\section{Introduction}
\subsection{Notations and inverse problems}
%
%

Let $n \in\mathbb{N}$. 
We consider an $n$-dimensional bounded domain $\Omega \subset \mathbb{R}^n$ with 
sufficiently smooth boundary $\pp\Omega$ (e.g., of $C^4$-class). 
Let $T>0$ and
set $Q=\Omega \times (0,T)$. 
Henceforth we use notations 
$\pp_t= \frac{\pp}{\pp t}$, $\pp_i=\frac{\pp}{\pp x_i}$, $i=1,2,\ldots, n$ 
, $\nabla = (\partial_1, \ldots, \partial_n)$ and $\pp_\nu =\nu \cdot \nabla$ 
where $\nu =\nu (x)$ is the outward unit normal vector to $\pp\Omega$ at $x$.

We consider the time-fractional diffusion equation of half-order in time:
\begin{equation}
\label{eq:fde0}
\begin{aligned}
& \pp_t^{\frac12} u(x,t) - A u(x,t)  = g(x,t), &&\quad (x,t)\in Q,
\end{aligned}
\end{equation}
where $A$ is the following elliptic operator:
\begin{equation*}
Au(x,t) = \dd (a(x)\nabla u(x,t)), \quad (x,t)\in Q,
\end{equation*} 
with $a \in C^3(\overline{\Omega})$ satisfying $\frac1{\mu}<a(x)<\mu$, $x\in \overline{\Omega}$ for a positive constant $\mu$,  
and $\pp_t^{\frac12}$ denotes the half-order Caputo type derivative defined by
\begin{equation*}
\pp_t^{\frac12}u(x,t) :=
\frac{1}{\Gamma\left(\frac12 \right)}\int_0^t\frac{\pp_\tau u(x,\tau)}{(t-\tau )^{\frac12}}
\,d\tau,
\quad (x,t)\in Q.
\end{equation*}

In this article, we investigate the following two inverse problems. 
Let $t_0 \in (0,T)$ be an arbitrarily fixed time.
We assume that $g(x,t)=f(x)R(x,t)$, $(x,t)\in Q$. 

\noindent
\textbf{Inverse Source Problem (ISP)}\\
Can we determine the spatially varying factor $f$ in the source of the half-order time-fractional diffusion equation \eqref{eq:fde0} from spatial data $u(\cdot, t_0)$? 

\noindent
\textbf{Inverse Coefficient Problem (ICP)}\\ 
Can we determine the diffusion coefficient $a$ of the half-order time-fractional diffusion equation \eqref{eq:fde0} from spatial data $u(\cdot, t_0)$?

\subsection{Motivation}
%
%

The equation \eqref{eq:fde0} appears 
as a macroscopic model describing the anomalous diffusion phenomena in heterogeneous media. 
Anomalous diffusion naturally arises from environmental issues 
such as soil contaminations.
Indeed, it has been shown by Adams and Gelhar \cite{Adams-Gelhar92} that the diffusion phenomena in the heterogeneous aquifer 
are not well explained by the classical diffusion equations. 
As a microscopic model of anomalous diffusion, 
Hatano and Hatano \cite{Hatano-Hatano98} introduced the continuous time random walk (CTRW) model. 
And it is known that fractional diffusion equations can be derived from the CTRW model (see e.g., \cite{Metzler-Klafter00}). 

%
%
Moreover, our work is highly motivated by the problems of 
identifying the source or the mass diffusivity, which are important for the prediction of the contamination. 
As a theoretical work, we mainly focus on the stability issues of such inverse problems in this article. 

\subsection{Known results}
%
%

Owing to the practical background, inverse problem for fractional diffusion equations has been a hot topic and there are many theoretical and numerical researches on it. For instance, we refer to the survey papers \cite{Liu-Li-Yamamoto19, Liu-Yamamoto19} and the references therein. 

However, as for the stability in such inverse problems, there are only several results. In one-dimensional case (i.e., $n=1$), Yamamoto and Zhang \cite{Yamamoto-Zhang12} established the stability estimate of H\"older type in some inverse problems for a half-order time fractional diffusion equation (see also \cite{Ren-Xu14}). Later Lipschitz type stability estimates in inverse source problems were given by Kawamoto \cite{Kawamoto18}. For multi-dimensional case ($n\ge 2$), Huang, Li and Yamamoto \cite{Huang-Li-Yamamoto19} recently succeeded in deriving a H\"older type stability in an inverse source problem for a special multi-term time-fractional diffusion equation including first-order time derivative. 

To the authors' best knowledge, our work is the first stability result in inverse problems for multi-dimensional fractional diffusion equations without adding a first-order time derivative. 



\subsection{Key method}

%
%
In order to obtain stability estimates in the inverse problems, 
we adopt the Bukhgeim-Klibanov (BK) method introduced in \cite{Bukhgeim-Klibanov81}. 
As for this method, we may refer to some monographs and survey papers \cite{Klibanov92, Klibanov13, Klibanov-Timonov04}. 
Actually BK method is based on so-called Carleman estimate which is an $L^2$-weighted {\it a priori} estimate for a partial differential equation and is a powerful tool in dealing with inverse problems and control problems.
For more details on the Carleman estimate, we refer to e.g., \cite{Fursikov-Imanuvilov96, Isakov06, Yamamoto09}. 

%
%

A Carleman estimate for a half-order time fractional diffusion equation was firstly derived by Xu, Cheng and Yamamoto \cite{Xu-Cheng-Yamamoto11} in one dimension. Here we follow their idea and reduce the equation \eqref{eq:fde0} to a multi-dimensional fourth-order partial differential equation. By establishing the Carleman estimate for this fourth-order differential equation, we manage to extend the results in \cite{Xu-Cheng-Yamamoto11} to a multi-dimensional version. As immediate applications of this Carleman estimate, we prove the stability results for our inverse problems by the BK method. 

We should mention that the first Carleman estimate for a fourth-order parabolic equation with constant coefficients in multi-dimension was derived by Guerrero and Kassab in \cite{Guerrero-Kassab19}. In the following context, required by the inverse coefficient problem, we modify their proof and establish the Carleman estimate in a more complicated case of variable coefficients. 


The rest of this article is organized as follows: 
In Section 2, we state the main results on the proposed inverse problems. 
Section 3 is devoted to the reduction from \eqref{eq:fde0} to a fourth-order partial differential equation and we also describe the Carleman estimate, which is a key tool in the proofs of our main results in Section 4. 
Finally, the proof of our key Carleman estimate is given in Section 5. 

%
%
\section{Main Results on Inverse Problems}

We choose $\delta>0$ such that $0< t_0-\delta<t_0<t_0+\delta< T$ and 
denote the time interval around $t_0$ by $I_\delta = (t_0-\delta,t_0+\delta)$. 

Taking $\Omega_0$ as a  sub-domain of $\Omega$ such that $\Omega_0\Subset\Omega$, that is, $\overline{\Omega_0} \subset \Omega$,  
we set $D=\Omega\setminus\overline{\Omega_0}$, 
then $D$ is a sub-domain around the boundary $\pp\Omega$. 

Define a function space $\mathcal{U}$ by
\begin{equation*}
\mathcal{U} := H^1(I_\delta; H^5(\Omega))\cap H^2(I_\delta; H^1(\Omega)).
\end{equation*}
Here and henceforth, $H^k(X) = W^{k,2}(X)$, $k\in \mathbb{N}$ where $W^{k,p}(X)$, $k,p\in \mathbb{N}$ denotes the Sobolev space on $X\subset (0,T)$ (or $X\subset \Omega$). 

%
%
\subsection{Stability in Inverse Source Problem}
Let $g(x,t)=f(x)R(x,t)$, $(x,t)\in Q$ in \eqref{eq:fde0}. 
In addition, we assume the zero initial condition:
\begin{equation}
\label{eq:ini0}
u(x,0)=0, \quad x\in\Omega,
\end{equation}
and we suppose that 
\begin{equation}
\label{eq:bd0}
u(x,t)=0,\quad (x,t)\in (D\cup\pp\Omega)\times I_\delta.  
\end{equation}

We assume that 
\begin{equation}
\label{eq:assumption_R}
R\in W^{2,\infty}(0,T; W^{2,\infty}(\Omega)), \quad
\left|R(x,t_0)\right|>0, \quad x\in \overline{\Omega}.
\end{equation}

Now we are ready to state our main result. 
\begin{thm}
\label{thm:isp}
We assume that $u \in \mathcal{U}$ satisfies \eqref{eq:fde0}, \eqref{eq:ini0}, \eqref{eq:bd0} and 
$R$ satisfies \eqref{eq:assumption_R}. Then there exists a constant $C>0$ such that 
\begin{equation}
\label{eq:isp}
\left\| f \right\|_{H^2(\Omega)} 
\leq 
C\left\| u(\cdot, t_0) \right\|_{H^4(\Omega)} 
\end{equation}
for any $f\in H^2(\Omega)$. 
\end{thm}
\begin{rmk}
In order to clarify our key idea and avoid some complicated calculations on the boundary integrals, 
here we assume strong assumption (3), that is, the solution vanishes in a neighborhood of the boundary 
$\partial\Omega$ over the time interval $I_\delta$. 
According to the proofs in the later sections, it is readily to see that without assuming (3), we can prove the stability estimate with an additional boundary measurement as follows:
$$
\left\| f \right\|_{H^2(\Omega)} 
\leq 
C\left\| u(\cdot, t_0) \right\|_{H^4(\Omega)} 
+ C\sum_{|\alpha|=0}^3 \|\partial_x^\alpha u\|_{H^1(I_\delta; L^2(\partial\Omega))}.
$$
Here and henceforth, we let $\partial_x^\alpha = \partial_1^{\alpha_1}\ldots\partial_n^{\alpha_n}$ for $\alpha = (\alpha_1, \ldots, \alpha_n)\in \mathbb{N}^n$ and $|\alpha| = \sum_{j=1}^n \alpha_j$. 
In fact, we conjecture that by treating the boundary integrals generated by integration by parts suitably, we could prove 
$$
\left\| f \right\|_{H^2(\Omega)} 
\leq 
C\left\| u(\cdot, t_0) \right\|_{H^4(\Omega)} 
+ C\|\partial_\nu u\|_{H^1(I_\delta; H^2(\Gamma))}
$$
for arbitrarily fixed sub-boundary $\Gamma\subset \partial\Omega$ 
provided that we assume homogeneous Dirichlet boundary condition instead of (3). 
\end{rmk}

%
%
\subsection{Stability in Inverse Coefficient Problem}

Let $u_0$ and $h$ be given functions. 
For $k=1,2$, we consider 
\begin{equation}
\label{eq:sys0}
\left\{
\begin{aligned}
\pp_t^{\frac12} u^{(k)}(x,t) - A^{(k)} u^{(k)}(x,t)  &= g(x,t), 	&(x,t)\in Q,\\
u^{(k)}(x,0)  &=u_0(x),  &x\in\Omega,\\
u^{(k)}(x,t)   &=h(x,t),  &(x,t)\in (D\cup\pp\Omega)\times I_\delta,  
\end{aligned}
\right.
\end{equation}
where 
\begin{equation*}
A^{(k)} u^{(k)}(x,t) =\dd (a^{(k)}(x)\nabla u^{(k)}(x,t)), \quad (x,t)\in Q,
\end{equation*} 
and 
$a^{(k)} \in C^4(\overline{\Omega})$ satisfying 
\begin{equation}
\label{eq:assumption_a}
\frac1{ \mu_0}< a^{(k)}(x) < \mu_0, \quad x\in\overline{\Omega}, 
\quad \text{and}\quad
\|a^{(k)}\|_{C^3(\overline{\Omega})} \leq M, 
\end{equation}
with some positive constants $\mu_0$ and $M$. 

By denoting $u(x,t)=u^{(1)}(x,t)-u^{(2)}(x,t), 
a(x)=a^{(1)}(x) - a^{(2)}(x), r(x,t)= u^{(2)}(x,t)$, $(x,t)\in Q$, we obtain from \eqref{eq:sys0} that
\begin{equation}
\label{eq:sys1}
\left\{
\begin{aligned}
\pp_t^{\frac12} u (x,t) - A^{(1)} u(x,t)  &= \dd (a(x) \nabla r(x,t)), 	&(x,t)\in Q,\\
u(x,0)&=0,	 &x\in\Omega,\\
u(x,t)&=0, 	&(x,t)\in (D\cup\pp\Omega)\times I_\delta,  
\end{aligned}
\right.
\end{equation}

In preparation to state the main result  on our inverse coefficient problem, 
we introduce the following weight functions. 
Set $Q_\delta=\Omega \times I_\delta$ and   
we take $d \in C^4(\overline{\Omega})$ such that
\begin{equation*}
d(x) > 0, \ x\in \Omega, \quad |\nabla d(x)| \ge \sigma, \ x\in\overline{\Omega}
\end{equation*}
where $\sigma>0$ is a positive constant. The existence of such function 
$d$ was shown, for example, in \cite{Fursikov-Imanuvilov96, Imanuvilov95, Imanuvilov-Yamamoto98}. 

Then we define $\va,\psi$ as follows:
\begin{equation*}
\va(x,t) := \frac{e^{\la d(x)}}{(t-t_0+\delta)(t_0+\delta-t)}, \quad 
\psi(x,t) := \frac{e^{\la d(x)}-e^{2\la \|d\|_{C(\overline{\Omega})}}}{(t-t_0+\delta)(t_0+\delta-t)}, \quad (x,t) \in Q_\delta.
\end{equation*}

Before stating the stability result, we assume that $r\in W^{2,\infty}(I_\delta; W^{5,\infty}(\Omega))$ such that
%
\begin{equation}
\label{eq:assumption_r}
\|r\|_{W^{2,\infty}(I_\delta; W^{5,\infty}(\Omega))} \le M, \quad 
|\nabla r(x,t_0)\cdot \nabla d(x)| \ge r_0, \ x\in \overline{\Omega_0}
\end{equation}
with a positive constant $r_0>0$. 

Now we state the second main result. 
\begin{thm}
\label{thm:icp}
We assume that $u^{(k)} \in \mathcal{U}$, $k=1,2$ satisfy \eqref{eq:sys0}. 
We suppose that 
$a^{(1)}, a^{(2)}\in C^4(\overline{\Omega})$ satisfy \eqref{eq:assumption_a}  
with $a^{(1)}= a^{(2)}$, $x\in D\cup\pp \Omega$
and $r=u^{(2)}$ satisfies \eqref{eq:assumption_r}. 
Then there exists a constant $C>0$ such that 
\begin{equation}
\label{eq:icp}
\left\| a^{(1)}-a^{(2)} \right\|_{H^3(\Omega)} 
\leq 
C\left\| u(\cdot, t_0) \right\|_{H^5(\Omega)}. 
\end{equation}
\end{thm}
\begin{rmk}
The second condition in \eqref{eq:assumption_r} is a mathematical assumption. 
It is common to impose such assumption while one intends to recover the coefficients 
under some differential operators such as the thermal conductivity in heat equation 
(see e.g., Theorem 6.1 in \cite{Yamamoto09}). 
\end{rmk}

%
%
\section{Carleman Estiamte}

Our main idea to derive the Carleman estimate for \eqref{eq:fde0} is
the reduction from the fractional diffusion equation to an integer-order equation, 
that is, a fourth-order partial differential equation. 
Although $\pp_t =\pp_t^{\frac12}\pp_t^{\frac12}$ does not hold for the Caputo type fractional derivative,  
we have the following lemma.
\begin{lem}
\label{lem:prelimi}
Assume that $u,\partial_t^{\frac12}u \in W^{1,1}(0,T; H^2(\Omega))$ 
satisfies \eqref{eq:fde0} and \eqref{eq:ini0}. 
Then we have the following equation
\begin{equation}
\label{eq:4thpara}
\left(\pp_t - A^2\right) u(x,t) = \left(\pp_t^{\frac{1}{2}}+A\right) g(x,t) + \frac{g(x,0)}{\sqrt{\pi t}}, 
\quad (x,t)\in Q.
\end{equation} 
\end{lem}
Here we used $\Gamma(\frac12)=\sqrt{\pi}$. 
This Lemma is a modified version of Lemma 2.3 in \cite{Xu-Cheng-Yamamoto11}. 
Applying the cofactor $\pp_t^{\frac12}+A$ to the equation \eqref{eq:fde0}, 
we can prove the above lemma by the similar argument used in \cite{Xu-Cheng-Yamamoto11}. 
Hereafter we consider the Carleman estimate for the fourth-order partial differential equation \eqref{eq:4thpara}. 

Now we are ready to state our Carleman estimate. 
\begin{thm}
\label{thm:ce0}
There exists $\la_0>0$ such that for any $\la>\la_0$, 
we can choose $s_0(\la)>0$ satisfying the following: 
there exists a constant $C=C(s_0,\la_0)>0$ such that 
\begin{align}
\label{eq:ce0}
&
\int_{Q_\delta} 
		\Biggl(
		\frac1{s\va} \left( |\pp_t u|^2 +|A^2 u|^2\right) 
		+
		s\la^2\va |\nabla A u|^2
		+ 
		s^3\la^4\va^3  |A u|^2
		\\&\qquad
		+ s^2\la^4\va^2  \sum_{i,j=1}^n |\pp_i \pp_j u |^2 
		+ s^4\la^6\va^4 |\nabla  u|^2 
		+ s^6\la^8\va^6 |u|^2
		\Biggr) e^{2s\psi}
\,dxdt
\nonumber\\&\leq
C\int_{Q_\delta} \left|(\pp_t - A^2)u\right|^2 e^{2s\psi} \,dxdt
\nonumber
\end{align}
for all $s\ge s_0$ and 
all $u \in L^2(I_\delta; H^4(\Omega))\cap H^1(I_\delta; L^2(\Omega))$ 
satisfying 
$\supp\, u \subset \Omega \times \overline{I_\delta}$. 
\end{thm}
\begin{rmk}
We may prove the above Carleman estimate with 
\begin{equation*}
Au(x,t)=\dd (a(x)\nabla u(x,t))
-\mathbf{b}(x)\cdot \nabla u(x,t)
-c(x)u(x,t), \quad (x,t)\in Q,
\end{equation*}
where 
$a \in C^3(\overline{\Omega})$, 
$\mathbf{b} \in \{C^2(\overline{\Omega})\}^n$, 
and $c \in C^2(\overline{\Omega})$. 
Indeed, the principal term 
is essential in the proof of Carleman estimate and
the lower-order terms could be absorbed by 
choosing sufficiently large parameters $s$ and $\la$. 
\end{rmk}

%
%
\section{Proof of Main Results}

%
%
\subsection{Proof of Main Result on Inverse Source Problem}
\begin{proof}[Proof of Theorem \ref{thm:isp}]
By Lemma \ref{lem:prelimi}, we have
\begin{equation}
\label{eq:prf1-1}
\left(\pp_t - A^2\right) u(x,t) = \left(\pp_t^{\frac{1}{2}}+A\right) g(x,t) + \frac{g(x,0)}{\sqrt{\pi t}} =: G(x,t) 
, \quad (x,t)\in Q.
\end{equation} 
Noting that $g(x,t)=f(x)R(x,t)$, we obtain
\begin{align}
\label{eq:prf1-2}
G(x,t)
&=
R(x,t) (Af)(x)
+
2 a(x)\nabla R(x,t) \cdot \nabla f(x)
\\&\quad
+
\left(
(AR)(x,t)
+
\pp_t^{\frac{1}{2}}R(x,t)
+ \frac{R(x,0)}{\sqrt{\pi t}}
\right)
f(x),
\quad (x,t)\in Q.
\nonumber
\end{align}

Taking weighted $L^2$-norm over $\Omega$ on both sides of \eqref{eq:prf1-1} at $t=t_0$, 
we have the following inequality. 
\begin{align}
\label{eq:prf1-3}
&
\int_\Omega
\left|
G(x,t_0)
\right|^2 e^{2s\psi(x,t_0)}
\,dx
\\&\leq
C
\int_\Omega
\left|
\pp_t u(x,t_0)
\right|^2 e^{2s\psi(x,t_0)}
\,dx
+
C
\int_\Omega
\sum_{|\alpha|\leq 4}
\left|
\pp^\alpha_x u(x,t_0)
\right|^2 e^{2s\psi(x,t_0)}
\,dx. 
\nonumber
\end{align}
Here and henceforth $C>0$ denotes a generic constant which may change line by line but is always independent of parameter $s$. 

First we estimate the first term on the right-hand side of \eqref{eq:prf1-3} 
by using our Carleman estimate (Theorem \ref{thm:ce0}). 
Second we estimate the term on the left-hand side of \eqref{eq:prf1-3} from below 
by the Calreman estimate for elliptic equations (Lemma \ref{lem:ce_elliptic1}). 
Henceforth we choose $s\ge 1$ large enough and fix $\lambda>0$ sufficiently large. 

Set $y=\pp_t u$. 
According to the choice of weight function, we readily see that $e^{2s\psi(x,t)} = 0$ for any $x\in \Omega$ as $t$ goes to $t_0-\delta$, which implies 
\begin{align}
\label{eq:prf1-4}
\int_\Omega
\left|
y (x,t_0)
\right|^2 e^{2s\psi(x,t_0)}
\,dx
&=
\int_{t_0-\delta}^{t_0} 
\pp_t \left( 
\int_\Omega
\left|
y
\right|^2 e^{2s\psi}
\,dx
\right)
\,dt
\\&=
\int_{t_0-\delta}^{t_0} 
\int_\Omega
\left(
2y \pp_t y+ 2s \pp_t \psi |y|^2
\right)
e^{2s\psi}
\,dx
\,dt
\nonumber\\&\leq
C
\int_{Q_\delta} 
|y||\pp_t y| e^{2s\psi}
\,dxdt
+
C
\int_{Q_\delta} 
s \va^2 |y| e^{2s\psi}
\,dxdt
.
\nonumber
\end{align}
Here we used that $\pp_t \psi(x,t)\leq C\va^2(x,t)$, $(x,t)\in Q_\delta$.  
Noting that 
\begin{equation*}
|y||\pp_t y|
=
\left(
s^{-\frac74}\va^{-\frac12} |\pp_t y|
\right)
\left(
s^{\frac74}\va^{\frac12} |y|
\right)
\leq 
Cs^{-\frac72}\va^{-1} |\pp_t y|^2
+
C s^{\frac72}\va |y|^2,  
\end{equation*}
together with $\varphi \le C\varphi^2$ and $s\ge 1$, from \eqref{eq:prf1-4} we obtain
\begin{align}
\label{eq:prf1-5}
&\int_\Omega \left|y (x,t_0) \right|^2 e^{2s\psi(x,t_0)}\,dx
\\&\le Cs^{-\frac52} \left( \int_{Q_\delta}\frac{1}{s\va} |\pp_t y|^2 e^{2s\psi} \,dxdt
+\int_{Q_\delta} s^6\va^6 |y|^2 e^{2s\psi}\,dxdt \right).
\nonumber
\end{align}

To estimate the right-hand side of \eqref{eq:prf1-5}, 
we need the Carleman estimate for $y$. 
Differentiating \eqref{eq:prf1-1} in $t$, we have
\begin{equation}
\label{eq:prf1-6}
\left(\pp_t - A^2\right) y(x,t) = \pp_t G(x,t), \quad (x,t)\in Q_\delta,
\end{equation} 
By \eqref{eq:bd0}, we see that $\supp\; y \in \Omega\times \overline{I_\delta}$. 
Applying Carleman estimate (Theorem \ref{thm:ce0}) 
to the equation \eqref{eq:prf1-6} 
yields
\begin{equation}
\label{eq:prf1-8}
\int_{Q_\delta} 
		\left(
		\frac1{s\va} |\pp_t y|^2 
		+ s^6\va^6 |y|^2
		\right) e^{2s\psi}
\,dxdt
\leq
C\int_{Q_\delta} \left|\pp_t G\right|^2 e^{2s\psi} \,dxdt. 
\end{equation}
Putting this together with \eqref{eq:prf1-5}, we have
\begin{equation*}
\int_\Omega
\left|
y (x,t_0)
\right|^2 e^{2s\psi(x,t_0)}
\,dx
\leq
Cs^{-\frac52} 
\int_{Q_\delta} \left|\pp_t G\right|^2 e^{2s\psi} \,dxdt,
\end{equation*}
that is, 
\begin{equation}
\label{eq:prf1-9}
\int_\Omega
\left|
\pp_t u (x,t_0)
\right|^2 e^{2s\psi(x,t_0)}
\,dx
\leq
Cs^{-\frac52} 
\int_{Q_\delta} \left|\pp_t G\right|^2 e^{2s\psi} \,dxdt. 
\end{equation}
Since
\begin{align*}
\pp_t G(x,t)
&=
\pp_tR(x,t) (Af)(x)
+
2 a\nabla \pp_t R(x,t) \cdot \nabla f(x)
\\&\quad
+
\left(
\pp_t (A R)(x,t)
+
\pp_t \pp_t^{\frac{1}{2}}R(x,t)
- \frac{R(x,0)}{2t\sqrt{\pi t}}
\right)
f(x),
\quad (x,t)\in Q_\delta,
\end{align*}
the inequality \eqref{eq:prf1-9} yields 
\begin{align*}
\int_\Omega
\left|
\pp_t u (x,t_0)
\right|^2 e^{2s\psi(x,t_0)}
\,dx
&\leq
Cs^{-\frac52} 
\int_{Q_\delta} 
\sum_{|\alpha|\leq 2}
\left|
\pp^\alpha_x f(x)
\right|^2 e^{2s\psi(x,t)}
\,dxdt
\\&\leq
Cs^{-\frac52} 
\int_{\Omega} 
\sum_{|\alpha|\leq 2}
\left|
\pp^\alpha_x f(x)
\right|^2 e^{2s\psi(x,t_0)}
\,dx. 
\end{align*}
Here we used $\psi(x,t)\leq \psi(x,t_0)$, $(x,t)\in Q_\delta$. 
Combining the above inequality with \eqref{eq:prf1-3}, we obtain
\begin{align}
\label{eq:prf1-10}
&
\int_\Omega
\left|
G(x,t_0)
\right|^2 e^{2s\psi(x,t_0)}
\,dx
\\&\leq
Cs^{-\frac52} 
\int_{\Omega} 
\sum_{|\alpha|\leq 2}
\left|
\pp^\alpha_x f(x)
\right|^2 e^{2s\psi(x,t_0)}
\,dx
+
C
\int_\Omega
\sum_{|\alpha|\leq 4}
\left|
\pp^\alpha_x u(x,t_0)
\right|^2 e^{2s\psi(x,t_0)}
\,dx. 
\nonumber
\end{align}

Next we estimate the left-hand side of \eqref{eq:prf1-10} from below. 
We may regard the equation \eqref{eq:prf1-2} at $t=t_0$: 
\begin{align}
\label{eq:prf1-11}
\frac{G(x,t_0)}{R(x,t_0)}
&=
 (Af)(x)
+
\frac{2 a(x)}{R(x,t_0)} \nabla R(x,t_0) \cdot \nabla f(x)
\\&\quad
+
\frac1{R(x,t_0)}
\left(
(AR)(x,t_0)
+
\pp_t^{\frac{1}{2}}R(x,t_0)
+ \frac{R(x,0)}{\sqrt{\pi t_0}}
\right)
f(x),
\quad x \in \Omega
\nonumber
\end{align}
as a second-order elliptic equation with respect to $f$, 
According to \eqref{eq:fde0} and \eqref{eq:bd0}, we have $f(x)R(x,t)=0$, $(x,t)\in (D\cup\pp\Omega)\times I_\delta$. 
In particular, $f(x)R(x,t_0)=0$, $x\in D\cup\pp\Omega$. 
By the assumption \eqref{eq:assumption_R}, we see that $f(x)=0$, $x\in D\cup\pp\Omega$. 
Hence we may use the Carleman estimate (Lemma \ref{lem:ce_elliptic1}) for \eqref{eq:prf1-11} and we obtain
\begin{equation}
\label{eq:prf1-12}
s^{-1}
\int_\Omega
\sum_{|\alpha|\leq 2}
\left|
\pp^\alpha_x f(x)
\right|^2 e^{2s\psi(x,t_0)}
\,dx
\leq 
C
\int_\Omega
\left|
G(x,t_0)
\right|^2 e^{2s\psi(x,t_0)}
\,dx.
\end{equation}
This, together with \eqref{eq:prf1-10}, indicates
\begin{align*}
&
s^{-1}
\int_\Omega
\sum_{|\alpha|\leq 2}
\left|
\pp^\alpha_x f(x)
\right|^2 e^{2s\psi(x,t_0)}
\,dx
\\&\leq
Cs^{-\frac52} 
\int_{\Omega} 
\sum_{|\alpha|\leq 2}
\left|
\pp^\alpha_x f(x)
\right|^2 e^{2s\psi(x,t_0)}
\,dx
+
C
\int_\Omega
\sum_{|\alpha|\leq 4}
\left|
\pp^\alpha_x u(x,t_0)
\right|^2 e^{2s\psi(x,t_0)}
\,dx. 
\end{align*}
Choosing sufficiently large $s\ge 1$, 
we may absorb the first term on the right-hand side of 
the above inequality into the left-hand side 
and we have
\begin{equation*}
s^{-1}
\int_\Omega
\sum_{|\alpha|\leq 2}
\left|
\pp^\alpha_x f(x)
\right|^2 e^{2s\psi(x,t_0)}
\,dx
\leq
C
\int_\Omega
\sum_{|\alpha|\leq 4}
\left|
\pp^\alpha_x u(x,t_0)
\right|^2 e^{2s\psi(x,t_0)}
\,dx. 
\end{equation*}
Fixing large $s\ge 1$, then $e^{2s\psi(x,t_0)}$ has a positive lower bound, which completes the proof of our stability estimate \eqref{eq:isp}. 
\end{proof}

%
%
\subsection{Proof of Main Result on Inverse Coefficient Problem}
\begin{proof}[Proof of Theorem \ref{thm:icp}]
With reference to \eqref{eq:sys1}, by Lemma \ref{lem:prelimi} we have
\begin{equation}
\label{eq:prf2-1}
\left(\pp_t - \left(A^{(1)}\right)^2\right) u(x,t) = H(x,t), \quad (x,t)\in Q,
\end{equation} 
where 
\begin{align}
\label{eq:prf2-2}
H(x,t)
&=
a^{(1)}(x) \nabla r(x,t) \cdot \nabla \Delta a(x) 
+2 a^{(1)}(x) \sum_{i,j=1}^n \left(\pp_i\pp_j r(x,t)\right) \left(\pp_i\pp_j a(x)\right)
\\&\quad
+a^{(1)}(x) \Delta r(x,t) \Delta a(x) 
+\nabla a^{(1)}(x) \cdot (\nabla r(x,t) \cdot \nabla) \nabla a(x)
\nonumber\\&\quad
+
\biggl(
\pp_t^{\frac12}\nabla r(x,t) 
+3a^{(1)}(x) \nabla \Delta r(x,t)
+(\Delta r(x,t))\nabla a^{(1)}(x)
\nonumber\\&\quad\quad
+(\nabla a^{(1)}(x)\cdot \nabla) \nabla r(x,t)
+\frac{\nabla r(x,0)}{\sqrt{\pi t}} 
\biggr)\cdot \nabla a(x)
\nonumber\\&\quad
+
\biggl(
\pp_t^{\frac{1}{2}} \Delta r(x,t)
+\nabla a^{(1)}(x) \cdot \nabla \Delta r(x,t)
\nonumber\\&\quad\quad
+a^{(1)}(x) \Delta^2 r(x,t)
+ \frac{\Delta r(x,0)}{\sqrt{\pi t}}
\biggr)
a(x),
\quad (x,t)\in Q.
\nonumber
\end{align}
We differentiate the equation \eqref{eq:prf2-1} with respect to $x$ in $Q_\delta$ to obtain
\begin{equation}
\label{eq:prf2-3}
\nabla \left(\pp_t - \left(A^{(1)}\right)^2\right) u(x,t) = \nabla H(x,t), \quad (x,t)\in Q_\delta,
\end{equation}
Taking the weighted $L^2$-norm over $\Omega$ on both sides of 
\eqref{eq:prf2-1} and \eqref{eq:prf2-3} at $t=t_0$, we find
\begin{align}
\label{eq:prf2-4}
&
\int_\Omega
\left(
\left|
H(x,t_0)
\right|^2 
+
\left|
\nabla H(x,t_0)
\right|^2 
\right)e^{2s\psi(x,t_0)}
\,dx
\\&\leq
C
\int_\Omega
\left(
\left|
\pp_t u(x,t_0)
\right|^2 
+
\left|
\nabla\pp_t u(x,t_0)
\right|^2 
\right)e^{2s\psi(x,t_0)}
\,dx
\nonumber\\&\quad
+
C
\int_\Omega
\sum_{|\alpha|\leq 5}
\left|
\pp^\alpha_x u(x,t_0)
\right|^2 e^{2s\psi(x,t_0)}
\,dx. 
\nonumber
\end{align}

By setting $y=\pp_t u$, 
we differentiate the equation \eqref{eq:prf2-1} to obtain 
\begin{equation}
\label{eq:prf2-5}
\left(\pp_t - \left(A^{(1)}\right)^2\right) y(x,t) = \pp_t H(x,t), \quad (x,t)\in Q_\delta.
\end{equation} 
From \eqref{eq:bd0}, we see that $y$ is compactly supported in $\Omega \times \overline{I_\delta}$. 
Then we can employ Theorem \ref{thm:ce0} to \eqref{eq:prf2-5} and we have
\begin{align}
\label{eq:prf2-7}
&\int_{Q_\delta} \left( \frac{1}{s\va} |\partial_t y|^2 + s^2\va^2 |\Delta y|^2+ s^4\va^4 |\nabla y|^2 +s^6\va^6 |y|^2 \right) e^{2s\psi}\,dxdt
\leq
C\int_{Q_\delta} \left|\pp_t H\right|^2 e^{2s\psi} \,dxdt.
\end{align}
Hereafter we assume $s\ge 1$, $\lambda>0$ are sufficiently large such that the Carleman estimates hold true. 

Next we estimate the first term on the right-hand side of \eqref{eq:prf2-4}. Similarly to the proof of Theorem \ref{thm:isp}, we readily find that the weight $e^{2s\psi(\cdot,t)}$ goes to $0$ as $t$ tends to $t_0-\delta$ and thus,
\begin{equation*}
\int_\Omega
\left|
y (x,t_0)
\right|^2 e^{2s\psi(x,t_0)}
\,dx
\leq
Cs^{-\frac52} \int_{Q_\delta} \left(\frac{1}{s\va} |\pp_t y|^2 + s^6\va^6 |y|^2 \right)e^{2s\psi} \,dxdt.
\end{equation*}
With more careful calculations, we estimate
\begin{align*}
\int_\Omega \left|\nabla y(x,t_0)\right|^2 e^{2s\psi(x,t_0)}\,dx 
&= \int_{t_0-\delta}^{t_0} \pp_t \left( \int_\Omega \left|\nabla y\right|^2 e^{2s\psi}\,dx \right)\,dt
\\&=
\int_{t_0-\delta}^{t_0} \int_\Omega 
\left( 2\nabla y\cdot \nabla \partial_t y+ 2s \pp_t \psi |\nabla y|^2\right) e^{2s\psi}
\,dxdt.
\end{align*}
Using integration by parts and noting that $\supp\, \partial_t y(\cdot,t) \subset \Omega$ for all $t\in I_\delta$, we find
\begin{align*}
&\int_{t_0-\delta}^{t_0} \int_\Omega (2\nabla y\cdot \nabla \partial_t y) e^{2s\psi}\,dxdt\\
&= 
\int_{t_0-\delta}^{t_0} \int_\Omega \left( -2\Delta y - 4s\nabla y\cdot \nabla \psi\right) (\partial_t y) e^{2s\psi}\,dxdt\\
&\le
C\int_{Q_\delta} \left(|\Delta y||\partial_t y| + s|\nabla y||\partial_t y|\right) e^{2s\psi}\,dxdt\\
&\le 
Cs^{-\frac12} \int_{Q_\delta} 
\left( \frac{1}{s\varphi}|\partial_t y|^2 + s^2\varphi^2|\Delta y|^2 + s^4\varphi^4|\nabla y|^2 \right) e^{2s\psi}
\,dxdt.
\end{align*}
Here we used H\"older's inequality, $\varphi \le C\varphi^2$ and $s\ge 1$ in the last line. Also we can readily check that 
\begin{align*}
&\int_{t_0-\delta}^{t_0} \int_\Omega 2s \pp_t \psi |\nabla y|^2 e^{2s\psi}\,dxdt
\le 
Cs^{-\frac12} \int_{Q_\delta} s^4\varphi^4|\nabla y|^2 e^{2s\psi}\,dxdt. 
\end{align*}
Combining the above inequalities with \eqref{eq:prf2-7}, we see that 
\begin{align}
\label{eq:prf2-8}
&\int_\Omega 
\left( \left|\pp_t u(x,t_0)\right|^2 +\left|\nabla \pp_t u(x,t_0)\right|^2 \right)e^{2s\psi(x,t_0)}
\,dx
\leq
Cs^{-\frac12}
\int_{Q_\delta} \left|\pp_t H\right|^2 e^{2s\psi} \,dxdt. 
\end{align}
Note that 
\begin{align*}
\pp_t H(x,t)
&=
a^{(1)}(x) \nabla \pp_t r(x,t) \cdot \nabla \Delta a(x) 
+2 a^{(1)}(x) \sum_{i,j=1}^n \left(\pp_i\pp_j \pp_t r(x,t)\right) \left(\pp_i\pp_j a(x)\right)
\\&\quad
+a^{(1)}(x) \Delta \pp_t r(x,t) \Delta a(x) 
+\nabla a^{(1)}(x) \cdot (\nabla \pp_t  r(x,t) \cdot \nabla) \nabla a(x)
\\&\quad
+
\biggl(
\pp_t \pp_t^{\frac12}\nabla r(x,t) 
+3a^{(1)}(x) \nabla \Delta \pp_t r(x,t)
+(\Delta \pp_t  r(x,t))\nabla a^{(1)}(x)
\nonumber\\&\quad\quad
+(\nabla a^{(1)}(x)\cdot \nabla) \nabla \pp_t  r(x,t)
-\frac{\nabla r(x,0)}{2t\sqrt{\pi t}} 
\biggr)\cdot \nabla a(x)
\\&\quad
+
\biggl(
\pp_t \pp_t^{\frac{1}{2}} \Delta r(x,t)
+\nabla a^{(1)}(x) \cdot \nabla \Delta \pp_t r(x,t)
\nonumber\\&\quad\quad
+a^{(1)}(x) \Delta^2 \pp_t r(x,t)
- \frac{\Delta r(x,0)}{2t\sqrt{\pi t}}
\biggr)
a(x),
\quad (x,t)\in Q_\delta.
\end{align*}

Hence we have 
\begin{align*}
&\int_\Omega
\left( \left|\pp_t u(x,t_0)\right|^2 + \left|\nabla \pp_t u(x,t_0)\right|^2 \right)e^{2s\psi(x,t_0)}
\,dx
\\&\leq
Cs^{-\frac12} \int_{Q_\delta} \sum_{|\alpha|\leq 3} |\pp_x^\alpha a(x)| e^{2s\psi(x,t)} \,dxdt
\\&\leq
Cs^{-\frac12} \int_{\Omega} \sum_{|\alpha|\leq 3} |\pp_x^\alpha a(x)| e^{2s\psi(x,t_0)} \,dx.  
\end{align*}
Putting this together with \eqref{eq:prf2-4}, we have
\begin{align}
\label{eq:prf2-9}
&
\int_\Omega
\left(
\left|H(x,t_0)\right|^2 
+
\left|\nabla H(x,t_0)\right|^2 
\right)e^{2s\psi(x,t_0)}
\,dx
\\&\leq
Cs^{-\frac12}
\int_{\Omega} 
\sum_{|\alpha|\leq 3} |\pp_x^\alpha a(x)|
e^{2s\psi(x,t_0)} \,dx
+
C
\int_\Omega
\sum_{|\alpha|\leq 5}
\left|
\pp^\alpha_x u(x,t_0)
\right|^2 e^{2s\psi(x,t_0)}
\,dx. 
\nonumber
\end{align}
We rewrite \eqref{eq:prf2-2} at $t=t_0$ to obtain
\begin{align*}
& a^{(1)}(x) \nabla r(x,t_0) \cdot \nabla \Delta a(x) 
\\&=
H(x,t_0)
-2 a^{(1)}(x) \sum_{i,j=1}^n \left(\pp_i\pp_j r(x,t_0)\right) \left(\pp_i\pp_j a(x)\right)
\\&\quad
-a^{(1)}(x) \Delta r(x,t_0) \Delta a(x) 
-\nabla a^{(1)}(x) \cdot (\nabla r(x,t_0) \cdot \nabla) \nabla a(x)
\\&\quad
-
\biggl(
\pp_t^{\frac12}\nabla r(x,t_0) 
+3a^{(1)}(x) \nabla \Delta r(x,t_0)
+(\Delta r(x,t_0))\nabla a^{(1)}(x)
\\&\quad\quad
+(\nabla a^{(1)}(x)\cdot \nabla) \nabla r(x,t_0)
+\frac{\nabla r(x,0)}{\sqrt{\pi t_0}} 
\biggr)\cdot \nabla a(x)
\\&\quad
-
\biggl(
\pp_t^{\frac{1}{2}} \Delta r(x,t_0)
+\nabla a^{(1)}(x) \cdot \nabla \Delta r(x,t_0)
\\&\quad\quad
+a^{(1)}(x) \Delta^2 r(x,t_0)
+ \frac{\Delta r(x,0)}{\sqrt{\pi t_0}}
\biggr) a(x),
\quad x\in \Omega. 
\end{align*}
By the assumption \eqref{eq:assumption_r} 
and $a\in C^4(\overline{\Omega})$ with $a(x)=0$, $x\in D\cup\pp \Omega$,  
we can apply Lemma \ref{lem:ce_3rdpde} and we obtain
\begin{align}
\label{eq:prf2-10}
&s \int_\Omega \sum_{|\alpha|\leq 3} |\pp_x^\alpha a(x)| e^{2s\psi(x,t_0)}\,dx
\\&\leq 
C
\int_\Omega
\left(
\left|H(x,t_0)\right|^2 + \left|\nabla H(x,t_0)\right|^2 
\right)e^{2s\psi(x,t_0)}
\,dx
\nonumber\\&\quad
+C
\int_\Omega \sum_{|\alpha|\leq 3} |\pp_x^\alpha a(x)| e^{2s\psi(x,t_0)}\,dx. 
\nonumber
\end{align}
Inserting \eqref{eq:prf2-9} into \eqref{eq:prf2-10} yields
\begin{align*}
&s \int_\Omega \sum_{|\alpha|\leq 3} |\pp_x^\alpha a(x)| e^{2s\psi(x,t_0)}\,dx
\\&\leq
C\int_{\Omega} \sum_{|\alpha|\leq 3} |\pp_x^\alpha a(x)| e^{2s\psi(x,t_0)} \,dx
+
C\int_\Omega 
\sum_{|\alpha|\leq 5} \left|\pp^\alpha_x u(x,t_0)\right|^2 e^{2s\psi(x,t_0)}
\,dx. 
\end{align*}
By fixing sufficiently large $s\ge 1$, 
we absorb the first term on the right-hand side of the above inequality 
into the left-hand side, which proves Theorem \ref{thm:icp}. 
\end{proof}

%
%
\section{Proof of the Carleman estimate}
As we know, Carleman estimate for the evolution equation $(\partial_t - K)u = F$, where $K$ is a differential operator of order higher than two, is not well discussed. 
Recently, Guerrero and Kassab \cite{Guerrero-Kassab19} established a Carleman estimate for the case $K = \Delta^2$. 
The key point is to find a suitable division of $Pw := e^{s\psi}(\partial_t - K)(e^{-s\psi}w)$ into $P_1w + P_2w + Rw$ such that one can recover positive terms from the inner product $(P_1w, P_2w)_{L^2}$.  

In this section, we consider a more complicated case $K=A^2$ where $A$ is a second-order partial differential operator with variable coefficients. For sake of simplicity, we do not discuss the boundary integrals in this paper and we assume that $u$ is compactly supported in $\Omega \times \overline{I_\delta}$, so that no spatial boundary terms appear when we carry on integration by parts. 

\begin{proof}
By setting $w = e^{s\psi}u$ and $Pw := e^{s\psi}(\partial_t - A^2)(e^{-s\psi}w)$, we obtain
\begin{align*}
Pw
&=
\pp_t w
- s(\pp_t \psi) w
%
%
-A^2 w 
+ 4 s\la \va a \nabla d \cdot \nabla A w 
%
%
\\&\quad
-  2 s^2\la^2 \va^2 a|\nabla d|^2 A w
- 4s^2\la^2\va^2 a^2 \nabla d \cdot  [(\nabla d\cdot \nabla )\nabla w] 
\\&\quad 
+ 2 s\la^2 \va a|\nabla d|^2 A w 
+ 4s\la^2 \va a^2 \nabla d \cdot [(\nabla d \cdot \nabla)\nabla w] 
\\&\quad 
- 2 s\la \va a (\nabla a \cdot \nabla d) \Delta w 
+ 6s\la \va a \nabla a \cdot [(\nabla d \cdot \nabla )\nabla w] 
\\&\quad 
+ 4s\la \va a^2 \sum_{i,j=1}^n (\pp_i\pp_j d)(\pp_i \pp_j w) 
+ 2 s\la \va (A d) A w
%
%
\\&\quad
+ 4s^3 \la^3 \va^3 a^2 |\nabla d|^2 \nabla d \cdot \nabla w
- 12s^2\la^3 \va^2 a^2|\nabla d|^2 \nabla d \cdot \nabla w 
\\&\quad 
- 4s^2 \la^2 \va^2 a (A d) \nabla d \cdot \nabla w
- 4s^2 \la^2 \va^2 a (\nabla a \cdot \nabla d)\nabla d \cdot \nabla w
\\&\quad
- 4s^2 \la^2 \va^2  a \nabla (a|\nabla d|^2) \cdot \nabla w
+ 2s^2 \la^2 \va^2 a |\nabla d|^2 \nabla a \cdot \nabla w
\\&\quad
+ 4s \la^3 \va a^2 |\nabla d|^2 \nabla d \cdot \nabla w
+ 4s\la^2 \va a (A d) \nabla d \cdot \nabla w 
\\&\quad 
+ 4s\la^2 \va a (\nabla a \cdot \nabla d) \nabla d\cdot \nabla w
+ 4s\la^2 \va a \nabla (a|\nabla d|^2) \cdot \nabla w
\\&\quad 
- 2s\la^2 \va a |\nabla d|^2 \nabla a \cdot \nabla w
+ 2s\la \va a^2 \nabla \Delta d \cdot \nabla w 
\\&\quad 
+ 2s\la \va a (\Delta a) \nabla d \cdot \nabla w 
+ 6s\la \va a \nabla a \cdot [(\nabla w \cdot \nabla )\nabla d] 
\\&\quad 
+ 2s\la \va |\nabla a|^2 \nabla d\cdot \nabla w
+ 2s\la \va a \nabla A d \cdot \nabla w
%
%
\\&\quad
- s^4 \la^4 \va^4 a^2 |\nabla d|^4 w
+ 6s^3 \la^4 \va^3 a^2 |\nabla d|^4 w
\\&\quad
+ 2s^3 \la^3 \va^3 a (A d)|\nabla d|^2 w
+2s^3\la^3\va^3 a [\nabla d \cdot \nabla (a|\nabla d|^2)] w
\\&\quad
-7s^2\la^4\va^2 a^2 |\nabla d|^4 w 
+ s\la^4 \va a^2 |\nabla d|^4 w
\\&\quad
- 6s^2\la^3 \va^2 a [\nabla d \cdot \nabla (a|\nabla d|^2)] w 
- 6s^2 \la^3 \va^2 a (A d)|\nabla d|^2 w
\\&\quad
- s^2 \la^2 \va^2 [A (a|\nabla d|^2)] w 
- s^2 \la^2 \va^2   (A d)^2 w
-2s^2\la^2\va^2 a [\nabla d \cdot \nabla (A d)] w
\\&\quad
+ 2s\la^3 \va a |\nabla d|^2 (A d) w
+ 2s\la^3 \va a [\nabla d \cdot \nabla (a|\nabla d|^2)] w 
+ s\la^2 \va (A d)^2 w 
\\&\quad
+ s \la^2 \va [A (a|\nabla d|^2)] w
+ 2s\la^2 \va a (\nabla d \cdot \nabla A d) w 
+ s \la \va (A^2 d) w 
\\&=P_1 w+ P_2 w + R w 
\end{align*}
where
\begin{align*}
P_1 w
:=
&\
\pp_t w
+ 4 s\la \va a \nabla d \cdot \nabla A w \\
&
+ 2 s\la^2 \va a|\nabla d|^2 A w 
+ 4s\la^2 \va a^2 \nabla d \cdot [(\nabla d \cdot \nabla)\nabla w] \\
&
+2s\la \va a (\nabla a \cdot \nabla d) \Delta w 
-4s\la \va a \nabla a \cdot [(\nabla d \cdot \nabla )\nabla w] \\
&
-4 s\la \va a^2 \sum_{i,j=1}^n (\pp_i\pp_j d)(\pp_i \pp_j w) 
+ 2 s\la \va (A d) A w \\
&
+ 4s^3 \la^3 \va^3 a^2 |\nabla d|^2 \nabla d \cdot \nabla w 
+ 6s^3 \la^4 \va^3 a^2 |\nabla d|^4 w\\
&
+ 2s^3 \la^3 \va^3 a (A d)|\nabla d|^2 w
+6s^3\la^3\va^3 a [\nabla d \cdot \nabla (a|\nabla d|^2)] w
,\\
P_2 w
:=
&
- A^2 w \\
&
-  2 s^2\la^2 \va^2 a|\nabla d|^2 A w
- 4s^2\la^2\va^2 a^2 \nabla d \cdot  [(\nabla d\cdot \nabla )\nabla w] \\
&
- 12s^2\la^3 \va^2 a^2|\nabla d|^2 \nabla d \cdot \nabla w
- 4s^2 \la^2 \va^2 a (A d) \nabla d \cdot \nabla w\\
&
- 4s^2 \la^2 \va^2 a (\nabla a \cdot \nabla d)\nabla d \cdot \nabla w
- 4s^2 \la^2 \va^2  a \nabla (a|\nabla d|^2) \cdot \nabla w\\
&
+ 2s^2 \la^2 \va^2 a |\nabla d|^2 \nabla a \cdot \nabla w 
- s^4 \la^4 \va^4 a^2 |\nabla d|^4 w
, \\
R w
= &\ P w - P_1 w - P_2  w. 
\end{align*}
In order to avoid confusions, we mention that for some technical reason, some terms in $Pw$ are divided into two parts and we put one part into $P_1w$ and the other part into $R w$. 
For example, the 9th term of $Pw$: $-2s\lambda\va a(\nabla a\cdot\nabla d)\Delta w$ is divided into
$$
-2s\lambda\va a(\nabla a\cdot\nabla d)\Delta w = 2s\lambda\va a(\nabla a\cdot\nabla d)\Delta w - 4s\lambda\va a(\nabla a\cdot\nabla d)\Delta w,
$$
and then we put $2s\lambda\va a(\nabla a\cdot\nabla d)\Delta w$ into $P_1w$ and $-4s\lambda\va a(\nabla a\cdot\nabla d)\Delta w$ into $Rw$. Also the 10th, 11th and 32nd terms of $Pw$ are similarly treated. 

Taking $L^2$-norm of $P w-R w$ in $Q_\delta$, we have
\begin{align*}
\| P w -R w \|_{L^2(Q_\delta)}^2
&=
\| P_1 w+P_2 w \|_{L^2(Q_\delta)}^2 
\\&=
   \| P_1 w\|_{L^2(Q_\delta)}^2
+ \| P_2 w\|_{L^2(Q_\delta)}^2 
\\&\quad
+ 2(P_1 w , P_2 w)_{L^2(Q_\delta)}.
\end{align*}
Moreover, since we note
\begin{equation*}
\| P w -Rw \|_{L^2(Q_\delta)}^2\leq 2\| Pw  \|_{L^2(Q_\delta)}^2 +2\|Rw \|_{L^2(Q_\delta)}^2, 
\end{equation*}
we obtain
\begin{equation}
\label{eq:prfC-1}
\frac12
\| P_1 w\|_{L^2(Q_\delta)}^2
+ (P_1 w , P_2 w)_{L^2(Q_\delta)} 
\leq 
\| Pw \|_{L^2(Q_\delta)}^2 +\| R w \|_{L^2(Q_\delta)}^2.
\end{equation}

As we mentioned at the beginning of this section, the main part of the proof is to estimate the inner product 
$
(P_1 w , P_2 w)_{L^2(Q_\delta)}
$ from below. 
Let $p_k$ $(k=0,\ldots, 11)$ be the $(k+1)$-th term in $P_1 w$ 
and let 
$q_\ell$ $(\ell=1,\ldots, 9)$ be the $\ell$-th term in $P_2 w$. 
Henceforth  by $I_{k,\ell}$ we denote the $L^2$-inner product $(p_k, q_\ell)_{L^2(Q_\delta)}$.

In order to obtain the Carleman estimate, we use integration by parts and we divide the proof into the following six steps. 
\begin{itemize}
\item 
Finding lower-order terms with respect to $s\va$ and $\la$ from $I_{k,\ell}$. 
\item 
Calculation of $I_{k,\ell}$ that includes time derivatives. 
\item 
Calculation of $I_{k,\ell}$ related to space derivatives. 
\item 
Summing up the inequalities derived above.
\item 
Recovery of terms of $(\nabla d \cdot \nabla)\nabla w$.
\item 
Rewriting the inequality in terms of $u$. 
\end{itemize}

Hereafter we assume that $\la \ge 1$ and $s\ge \max\{1,\delta^2\}$, so that $s\va \ge 1$ in $Q_\delta$. 
Furthermore, we introduce the lower-order terms $\low$ 
which will be absorbed into the left-hand side of the Carleman estimate 
by taking sufficiently large $s$ and $\la$ in the last step. 
\begin{align*}
\low
:\!\!&=
	\int_{Q_\delta} (s\la\va+s^{\frac12}\la^2\va^{\frac12})  |\nabla d \cdot\nabla A w|^2\,dxdt\\
&\quad
+	\int_{Q_\delta} (s^3\la^3\va ^3+ s^{\frac52}\la^4\va^{\frac52})  \left|(\nabla d \cdot \nabla)\nabla w\right|^2\,dxdt\\
&\quad
+	\int_{Q_\delta} (s^5\la^5\va ^5+ s^{\frac92}\la^6\va^{\frac92}) |\nabla d \cdot \nabla w|^2\,dxdt\\
&\quad
+	\int_{Q_\delta} (s\la\va+s^{\frac12}\la^2\va^{\frac12})  |\nabla A u|^2 e^{2s\psi}\,dxdt\\
&\quad
+ \int_{Q_\delta} (s^3\la^3\va ^3+ s^{\frac52}\la^4\va^{\frac52}) |A u|^2 e^{2s\psi}\,dxdt\\
&\quad
+ \int_{Q_\delta} (s^2\la^3\va ^2+ s^{\frac32}\la^4\va^{\frac32}) \sum_{i,j=1}^n|\pp_i\pp_j u|^2 e^{2s\psi}\,dxdt\\
&\quad
+	\int_{Q_\delta} (s^4\la^5\va ^4+ s^{\frac72}\la^6\va^{\frac72})) |\nabla u|^2 e^{2s\psi}\,dxdt\\
&\quad
+	\int_{Q_\delta} (s^6\la^7\va ^6+ s^{\frac{11}2}\la^8\va^{\frac{11}2}) | u|^2 e^{2s\psi}\,dxdt
.
\end{align*}
Since $w= e^{s\psi}u$, we have $\nabla w = e^{s\psi}(\nabla u + s\lambda\varphi u \nabla d)$. Then by triangle inequality, we note that
\begin{align*}
&\int_{Q_\delta} (s^4\la^5\va^4+ s^{\frac72}\la^6\va^{\frac72}) |\nabla w|^2\,dxdt + \int_{Q_\delta} (s^6\la^7\va^6+ s^{\frac{11}2}\la^8\va^{\frac{11}2})  |w|^2\,dxdt 
\leq C\low.
\end{align*}
Similarly, we obtain
\begin{align*}
&\int_{Q_\delta} (\la+ s^{-\frac12}\la^2\va^{-\frac12}) |\nabla A w|^2 \,dxdt \le C\low,\\
&\int_{Q_\delta} (s^2\la^3\va^2+ s^{\frac32}\la^4\va^{\frac32})
\left( |A w|^2 +\sum_{i,j=1}^n|\pp_i\pp_j w|^2\right)\,dxdt \leq C\low.
\end{align*}

Then we can easily check that 
\begin{equation}
\label{eq:prfC-2}
\|Rw \|_{L^2(Q_\delta)}^2
\leq C\low.
\end{equation}

%
%
\noindent
\textbf{Step 1: Finding lower-order terms with respect to $s\va$ and $\la$.  }

We pick up several terms from $I_{k,\ell}$ ($k=0, \ldots, 11$, $\ell=1,\ldots, 9$) which can be directly estimated by $\low$. By using the inequality, $2\xi\eta\leq \xi^2+\eta^2$, we have the lower bound of the following 40 terms:
\begin{equation*}
	\sum_{\ell=4}^8 \sum_{k=2}^7 I_{k,\ell} 
+ 	\sum_{\ell=4}^8 \sum_{k=10}^{11} I_{k,\ell}  
\geq -C\low. 
\end{equation*}

%
%
\noindent
\textbf{Step 2: Calculation of $I_{k,\ell}$ that includes time derivatives. }

We compute 9 terms from $I_{0,1}$ to $I_{0,9}$ by integration by parts. 
According to the choice of weight function, we find $\psi(\cdot,t)\to -\infty$ as $t\to t_0 \pm \delta$, 
which implies $w(\cdot,t)\to 0$ as $t\to t_0 \pm \delta$. Hence, we do not have any temporal boundary terms as we carry on integration by parts. By direct calculation, we have
\begin{align*}
I_{0,1}
&=
-\int_{Q_\delta}
(\pp_t w)\dd (a \nabla A w)
\,dxdt
=0.
\\
I_{0,2}
&=
-\int_{Q_\delta}
2s^2\la^2\va^2a |\nabla d|^2 (\pp_t w)\dd (a \nabla w)
\,dxdt
\\&=
-\frac13 I_{0,4}-\frac12 I_{0,7}
-
\int_{Q_\delta}
2s^2\la^2\va (\pp_t \va) a^2|\nabla d|^2 |\nabla w|^2
\,dxdt.
\\
I_{0,3}
&=
-\int_{Q_\delta}
4s^2\la^2\va^2a^2\sum_{i,j=1}^n (\pp_i d)(\pp_j d) (\pp_i \pp_j w) \pp_t w
\,dxdt
\\&=
-\frac23 I_{0,4}
-I_{0,5}
-I_{0,6}
-\frac12 I_{0,7}
-I_{0,8}
-
\int_{Q_\delta}
4s^2\la^2\va (\pp_t \va)a^2 |\nabla d \cdot \nabla w|^2
\,dxdt.
\\
I_{0,4}
&=
-\int_{Q_\delta}
12s^2\la^3\va^2a^2|\nabla d|^2(\nabla  d\cdot \nabla w) \pp_t w
\,dxdt.
\\
I_{0,5}
&=
-\int_{Q_\delta}
4s^2\la^2\va^2a (A d)(\nabla d\cdot \nabla w) \pp_t w
\,dxdt.
\\
I_{0,6}
&=
-\int_{Q_\delta}
4s^2\la^2\va^2a(\nabla a \cdot \nabla d)(\nabla d \cdot \nabla w) \pp_t w
\,dxdt.
\\
I_{0,7}
&=
-\int_{Q_\delta}
4s^2\la^2\va^2a \left[ \nabla (a |\nabla d|^2)\cdot \nabla w \right] \pp_t w
\,dxdt.
\\
I_{0,8}
&=
\int_{Q_\delta}
2s^2\la^2\va^2a |\nabla d|^2(\nabla a \cdot \nabla w ) \pp_t w
\,dxdt.
\\
I_{0,9}
&=
-\int_{Q_\delta}
s^4\la^4\va^4 a^2|\nabla d|^4 w \pp_t w
\,dxdt
=
\int_{Q_\delta}
2s^4\la^4\va^3 (\pp_t \va) a^2|\nabla d|^4 |w|^2
\,dxdt.
\end{align*}
Summing up the above equalities, we obtain
\begin{align*}
I_0
&:=\sum_{\ell=1}^9 I_{0,\ell}\\
&=
-
\int_{Q_\delta}
2s^2\la^2\va (\pp_t \va) a^2|\nabla d|^2 |\nabla w|^2
\,dxdt
-
\int_{Q_\delta}
4s^2\la^2\va (\pp_t \va)a^2 |\nabla d \cdot \nabla w|^2
\,dxdt
\\&\quad
+
\int_{Q_\delta}
2s^4\la^4\va^3 (\pp_t \va) a^2|\nabla d|^4 |w|^2
\,dxdt
\\
&\geq
-C\low. 
\end{align*}

%
%
\noindent
\textbf{Step 3: Calculation of $I_{k,\ell}$ related to space derivatives.}

Again in terms of integration by parts, we calculate the remaining 59 terms. 

\noindent
\textbf{(\romannumeral1) Calculations of $I_{k,\ell}$ which give us 3rd derivatives.}

We first consider the 7 terms $I_{1,1}, \ldots, I_{7,1}$ and denote the sum of them by 
\begin{equation*}
I_1:=\sum_{k=1}^7 I_{k,1}.
\end{equation*} 
By integration by parts, 
we have 
\begin{align*}
I_{1,1}
&=
-
\int_{Q_\delta}
4s\la\va a(\nabla d \cdot \nabla A w) \dd (a \nabla A w)
\,dxdt 
\\&=
\int_{Q_\delta}
4s\la^2\va a^2|\nabla d \cdot \nabla A w|^2
\,dxdt
-
\int_{Q_\delta}
2s\la^2\va a^2  |\nabla d|^2 |\nabla A w|^2
\,dxdt
\\&\quad
+
\int_{Q_\delta}
4s\la\va a(\nabla d \cdot \nabla A w) \nabla a \cdot \nabla A w
\,dxdt 
\\&\quad
+
\int_{Q_\delta}
4s\la\va a^2 [(\nabla A w \cdot \nabla )\nabla d]\cdot  \nabla A w
\,dxdt
\\&\quad
-
\int_{Q_\delta}
2s\la\va a  (\nabla a \cdot \nabla d) |\nabla A w|^2
\,dxdt
-
\int_{Q_\delta}
2s\la\va a (A d) |\nabla A w|^2
\,dxdt.
\\
I_{2,1}
&=
-
\int_{Q_\delta}
2s\la^2\va a|\nabla d|^2 (A w)A^2 w
\,dxdt
\\&\geq
-C\low
+
\int_{Q_\delta}
2s\la^2\va a^2|\nabla d|^2 |\nabla A w|^2
\,dxdt.
\\
I_{3,1}
&=
-
\int_{Q_\delta}
4s\la^2\va a^2 \left( \nabla d \cdot [(\nabla d \cdot \nabla)\nabla w] \right) A^2 w
\,dxdt
\\&\geq
-C\low
+
\int_{Q_\delta}
4s\la^2\va a^2  |\nabla d \cdot \nabla A w|^2 
\,dxdt.
\\
I_{4,1}
&=
-
\int_{Q_\delta}
2s\la\va a (\nabla a \cdot \nabla d) (\Delta w) A^2 w
\,dxdt
\\&\geq
-C\low+
\int_{Q_\delta}
2s\la\va a (\nabla a \cdot \nabla d)  |\nabla A w|^2
\,dxdt.
\\
I_{5,1}
&=
\int_{Q_\delta}
4s\la\va a \left( \nabla a \cdot [(\nabla d \cdot \nabla)\nabla w] \right) A^2 w
\,dxdt
\\&\geq
-C\low
-
\int_{Q_\delta}
4s\la\va a  \left(\nabla a \cdot \nabla A w \right)\nabla d \cdot \nabla A w 
\,dxdt.
\\
I_{6,1}
&=
\int_{Q_\delta}
4s\la\va a^2 \left( \sum_{i,j=1}^n (\pp_i \pp_j d)(\pp_i \pp_j w) \right) A^2 w
\,dxdt
\\&\geq
-C\low
-
\int_{Q_\delta}
4s\la\va a^2  \nabla A w \cdot [(\nabla A w \cdot \nabla) \nabla d]
\,dxdt.
\\
I_{7,1}
&=
-
\int_{Q_\delta}
2s\la\va (A d)(A w)A^2 w
\,dxdt
\\&\geq
-C\low
+
\int_{Q_\delta}
2s\la\va a (A d)|\nabla A w|^2
\,dxdt
\end{align*}

Thus we obtain
\begin{equation*}
I_1
\geq
-C\low
+
\int_{Q_\delta}
8s\la^2\va a^2  |\nabla d \cdot \nabla A w|^2 
\,dxdt
.
\end{equation*}

\noindent
\textbf{(\romannumeral2) Calculations of $I_{k,\ell}$ which give us 2nd derivatives.}

Next we will calculate the following 23 terms: 
$I_{1,2}, \ldots, I_{1,8}$, 
$I_{2,2}, \ldots, I_{7,2}$, 
$I_{2,3}, \ldots, I_{7,3}$, 
$I_{8,1}, \ldots, I_{11,1}$. 
The sum of these terms is denoted by
\begin{align*}
I_2
&:=\sum_{\ell=2}^8 I_{1, \ell}
+\sum_{k=2}^7 I_{k,2}
+\sum_{k=2}^7 I_{k,3}
+\sum_{k=8}^{11} I_{k,1}. 
\end{align*}
In order to make the calculations more comprehensible, we divide $I_2$ into three parts in which similar integrals appear: 
\begin{equation*}
I_2
=
I_{2,a}
+
I_{2,b}
+
I_{2,c}
\end{equation*}
where $I_{2,a},I_{2,b},I_{2,c}$ are given by  
\begin{align*}
I_{2,a}&=
I_{2,2}+I_{4,2}+I_{7,2}
+
I_{1,2}+I_{2,3}+I_{3,2}+I_{5,2}+I_{6,2}\\&\quad+I_{9,1}+I_{10,1}+I_{11,1}
+
I_{1,8}+I_{1,4}+I_{1,7}+I_{8,1},\\
I_{2,b}&=
I_{3,3}+I_{4,3}+I_{5,3}+I_{6,3}
+
I_{1,3}+I_{1,6},\\
I_{2,c}&=
I_{7,3}
+
I_{1,5}
.
\end{align*}
First, we start from the calculation of $I_{2,a}$. 
\begin{align*}
I_{2,2}
&=
-
\int_{Q_\delta}
4s^3\la^4\va^3 a^2 |\nabla d|^4 |A w|^2
\,dxdt.
\\
I_{4,2}
&=
-
\int_{Q_\delta}
4s^3\la^3\va^3 a^2 (\nabla a \cdot \nabla d) |\nabla d|^2 (\Delta w) A w
\,dxdt
\\&\geq
-C\low 
-
\int_{Q_\delta}
4s^3\la^3\va^3 a (\nabla a \cdot \nabla d) |\nabla d|^2 |A w|^2
\,dxdt.
\\
I_{7,2}
&=
-
\int_{Q_\delta}
4s^3\la^3\va^3 a (A  d ) |\nabla d|^2|A w|^2
\,dxdt .
\end{align*}
And then, we have
\begin{align*}
I_{1,2}
&=
-
\int_{Q_\delta}
8s^3\la^3\va^3 a^2 |\nabla d|^2 (\nabla d \cdot \nabla A w) A w
\,dxdt 
\\&\geq
-C\low
-3I_{2,2}
-I_{4,2}
-I_{7,2}
\\&\quad
+
\int_{Q_\delta}
8s^3\la^3\va^3 a^2 \left( \nabla d \cdot [(\nabla d \cdot \nabla )\nabla d]\right) |A w|^2
\,dxdt.
\\
-I_{2,3}&=-I_{3,2}
=
\int_{Q_\delta}
8 s^3\la^4\va^3 a^3 |\nabla d|^2 (A w) \nabla d \cdot [(\nabla d\cdot \nabla) \nabla w]
\,dxdt
\\&\geq
-C\low
+
\int_{Q_\delta}
8 s^3\la^4\va^3 a^4 |\nabla d|^2   |(\nabla d\cdot \nabla) \nabla w|^2  
\,dxdt.
\\
I_{5,2}
&=
\int_{Q_\delta}
8s^3\la^3\va^3 a^2 |\nabla d|^2 \left( \nabla a \cdot [(\nabla d \cdot \nabla)\nabla w]\right)A w
\,dxdt.
\\
I_{6,2}
&=
\int_{Q_\delta}
8s^3\la^3\va^3 a^3 |\nabla d|^2 \left( \sum_{i,j=1}^n (\pp_i \pp_j d) (\pp_i \pp_j w)\right) A w
\,dxdt.
\\
I_{9,1}
&=
-
\int_{Q_\delta}
6s^3\la^4\va^3 a^2 |\nabla d|^4 w A^2 w
\,dxdt
\geq
-C\low
+\frac32 I_{2,2}.
\\
I_{10,1}
&=
-
\int_{Q_\delta}
2s^3\la^3\va^3 a (A d) |\nabla d|^2 w A^2 w
\,dxdt
\geq
-C\low
+\frac12 I_{7,2}.
\\
I_{11,1}
&=
-
\int_{Q_\delta}
6s^3\la^3\va^3 a[ \nabla d \cdot \nabla (a|\nabla d|^2)] w A^2 w
\,dxdt
\\&\geq
-C\low
+\frac32 I_{4,2}
-
\int_{Q_\delta}
12s^3\la^3\va^3 a^2 \left(\nabla d \cdot [ (\nabla d \cdot \nabla)\nabla d]\right)  |A w|^2
\,dxdt.
\end{align*}
Furthermore, we see that
\begin{align*}
-I_{1,8}
&=
-
\int_{Q_\delta}
8s^3\la^3\va^3 a^2 |\nabla d|^2 (\nabla a \cdot \nabla w) \nabla d \cdot \nabla A w
\,dxdt
\geq
-C\low 
+I_{5,2}.
\\
I_{1,4}
&=
-
\int_{Q_\delta}
48s^3\la^4\va^3 a^3 |\nabla d|^2(\nabla d \cdot \nabla w) \nabla d \cdot \nabla A w
\,dxdt
\geq
-C\low -6I_{2,3}.
\\
I_{1,7}
&=
-
\int_{Q_\delta}
16s^3\la^3\va^3 a^2 \left[ \nabla (a|\nabla d|^2)\cdot \nabla w \right]\nabla d \cdot \nabla A w
\,dxdt
\\&\geq
-C\low
-2I_{1,8}
+
\int_{Q_\delta}
32s^3\la^3\va^3 a^3 \left([(\nabla d \cdot \nabla)\nabla d]\cdot [(\nabla d \cdot \nabla)\nabla w]\right) A w
\,dxdt.
\\
I_{8,1}
&=
-
\int_{Q_\delta}
4s^3\la^3\va^3 a^2 |\nabla d|^2 (\nabla d \cdot \nabla w)A^2 w
\,dxdt
\\&\geq
-C\low
+\frac12 I_{1,2}
-\frac14 I_{1,4}
+\frac32 I_{2,3}
-I_{4,2}
-2 I_{5,2}
- I_{6,2}
\\&\quad
-
\int_{Q_\delta}
16s^3\la^3\va^3 a^3  \left( [(\nabla d \cdot \nabla)\nabla d] \cdot  [(\nabla d \cdot \nabla ) \nabla w] \right) A w
\,dxdt.
\end{align*}
Hence we have
\begin{align*}
I_{2,a}
&\geq-C\low
+
\int_{Q_\delta}
8s^3\la^4\va^3 a^2 |\nabla d|^4 |A w|^2
\,dxdt
\\&\quad
+
\int_{Q_\delta}
8 s^3\la^4\va^3 a^4 |\nabla d|^2   |(\nabla d\cdot \nabla) \nabla w|^2  
\,dxdt
\\&\quad
+
\int_{Q_\delta}
16s^3\la^3\va^3 a^3 \left([(\nabla d \cdot \nabla)\nabla d]\cdot [(\nabla d \cdot \nabla)\nabla w]\right) A w
\,dxdt.
\end{align*} 
Second, we compute $I_{2,b}$. 
\begin{align*}
I_{3,3}
&=
-
\int_{Q_\delta}
16 s^3\la^4\va^3 a^4  |\nabla d \cdot [(\nabla d \cdot \nabla) \nabla w]|^2
\,dxdt.
\\
I_{4,3}
&=
-
\int_{Q_\delta}
8 s^3\la^3\va^3 a^3  (\nabla a \cdot \nabla d) (\Delta w)\nabla d \cdot [(\nabla d \cdot \nabla) \nabla w]
\,dxdt 
\geq -C\low.
\\
I_{5,3}
&=
\int_{Q_\delta}
16 s^3\la^3\va^3 a^3  (\nabla a \cdot [(\nabla d\cdot \nabla) \nabla w]) \nabla d \cdot [(\nabla d \cdot \nabla) \nabla w]
\,dxdt 
\geq -C\low.
\\
I_{6,3}
&=
\int_{Q_\delta}
16s^3\la^3\va^3 a^4  \left( \sum_{i,j=1}^n (\pp_i\pp_j d)(\pp_i \pp_j w) \right) \nabla d \cdot [(\nabla d \cdot \nabla) \nabla w]
\,dxdt 
\geq
-C\low.
\\
I_{1,3}
&=
-
\int_{Q_\delta}
16 s^3\la^3\va^3 a^3  (\nabla d \cdot \nabla A w) \nabla d \cdot [(\nabla d \cdot \nabla) \nabla w]
\,dxdt
\\&\geq
-C\low
-3 I_{3,3}
+2I_{4,3}
+I_{6,3}
+3I_{5,3}
-
\int_{Q_\delta}
24 s^3\la^4\va^3 a^4  |\nabla d|^2  |(\nabla d \cdot \nabla ) \nabla w|^2
\,dxdt
\\&\quad
-
\int_{Q_\delta}
16 s^3\la^3\va^3 a^3  (A w) [(\nabla d \cdot \nabla)\nabla d] \cdot [(\nabla d \cdot \nabla )\nabla w]
\,dxdt.
\\
I_{1,6}
&=
-
\int_{Q_\delta}
16s^3\la^3\va^3 a^2 (\nabla a \cdot \nabla d) (\nabla d \cdot \nabla w) \nabla d \cdot \nabla A w
\,dxdt
\geq
-C\low
-2I_{4,3}.
\end{align*}
Summing up the above inequalities, we obtain
%
\begin{align*}
I_{2,b}
&\geq
-C\low
-
\int_{Q_\delta}
24 s^3\la^4\va^3 a^4  |\nabla d|^2  |(\nabla d \cdot \nabla ) \nabla w|^2
\,dxdt
\\&\quad
+
\int_{Q_\delta}
32 s^3\la^4\va^3 a^4  |\nabla d \cdot [(\nabla d \cdot \nabla) \nabla w]|^2
\,dxdt
\\&\quad
-
\int_{Q_\delta}
16 s^3\la^3\va^3 a^3 [(\nabla d \cdot \nabla)\nabla d] \cdot [(\nabla d \cdot \nabla )\nabla w] A w
\,dxdt. 
\end{align*}
Third, we calculate $I_{2,c}$. 
\begin{align*}
I_{1,5}
&=
-
\int_{Q_\delta}
16s^3\la^3\va^3a^2 (Ad) (\nabla d \cdot \nabla w)(\nabla d \cdot \nabla A w)
\,dxdt
\geq -C\low.
\\
I_{7,3}
&=
-\int_{Q_\delta}
8s^3\la^3\va^3 a^2(A d)(A w)\nabla d \cdot [(\nabla d \cdot \nabla)\nabla w]
\,dxdt
\geq -C\low -\frac12 I_{1,5}.
\end{align*}

Hence we have
\begin{equation*}
I_{2,c} \geq -C\low.
\end{equation*}
Thus we get the following inequality as the result of the above calculations. 
\begin{align*}
I_2
&=
I_{2,a}
+
I_{2,b}
+
I_{2,c}
\\&\geq
-C\low
+
\int_{Q_\delta}
8s^3\la^4\va^3 a^2 |\nabla d|^4 |A w|^2
\,dxdt
\\&\quad
-
\int_{Q_\delta}
16 s^3\la^4\va^3 a^4  |\nabla d|^2  |(\nabla d \cdot \nabla ) \nabla w|^2
\,dxdt.
\end{align*}
In fact, on the right-hand side we have another non-negative term: 
$$
\int_{Q_\delta} 32s^3\lambda^4\varphi^3 a^4 |\nabla d\cdot [(\nabla d\cdot\nabla)\nabla w]|^2\,dxdt.
$$
However, we do not put it in the estimate of $I_2$ since it is not used in the following contexts. 

\noindent
\textbf{(\romannumeral3) Calculations of $I_{k,\ell}$ which give us 1st derivatives.}

We compute the following $20$ terms:  
$I_{1,9}, \ldots, I_{7,9}$, 
$I_{8,2}, \ldots, I_{8,8}$, 
$I_{9,2}, \ldots, I_{11,2}$, 
$I_{9,3},$ $\ldots,$ $I_{11,3}$. 
The sum of these terms is denoted by
\begin{align*}
I_3
&:=
\sum_{k=1}^7 I_{k,9}
+\sum_{\ell=2}^8 I_{8, \ell}
+\sum_{k=9}^{11} I_{k,2}
+\sum_{k=9}^{11} I_{k,3}
=: I_{3,a}+I_{3,b}, 
\end{align*}
where $I_{3,a},I_{3,b}$ are the sums of terms given by 
\begin{align*}
I_{3,a}&=
I_{8,4}+I_{8,8}
+
I_{3,9}+I_{5,9}+I_{8,3}+I_{8,7}+I_{9,3}
\\&\quad
+
I_{2,9}+I_{4,9}+I_{7,9}+I_{10,2}+I_{8,2}+I_{9,2}
+
I_{1,9}.
\\
I_{3,b}&=
I_{6,9}
+
I_{8,5}
+
I_{8,6}
+
I_{10,3}
+
I_{11,2}
+
I_{11,3}
.
\end{align*}
We note that
\begin{align*}
I_{8,4}
&=
-
\int_{Q_\delta}
48s^5\la^6\va^5a^4 |\nabla d|^4 |\nabla d \cdot \nabla w|^2
\,dxdt.
\\
I_{8,8}
&=
\int_{Q_\delta}
8s^5\la^5\va^5a^3 |\nabla d|^4 (\nabla a \cdot \nabla w)  \nabla d \cdot \nabla w
\,dxdt.
\end{align*}
Then we have
\begin{align*}
I_{3,9}
&=
-
\int_{Q_\delta}
4s^5\la^6\va^5 a^4 |\nabla d|^4 w \nabla d \cdot [(\nabla d \cdot \nabla) \nabla w]
\,dxdt
\geq
-C\low
-\frac1{12} I_{8,4}. 
\\
I_{5,9}
&=
\int_{Q_\delta}
4s^5\la^5\va^5 a^3 |\nabla d|^4 w \nabla a \cdot [(\nabla d \cdot \nabla ) \nabla w ]
\,dxdt
\geq
-C\low
-\frac12 I_{8,8}.
\\
I_{8,3}
&=
-
\int_{Q_\delta}
16s^5\la^5\va^5 a^4 |\nabla d|^2 (\nabla d \cdot \nabla w) \nabla d\cdot [(\nabla d \cdot \nabla )\nabla w]
\,dxdt
\\&\geq
-C\low
-\frac56 I_{8,4}
+
\int_{Q_\delta}
16s^5\la^5\va^5 a^4 |\nabla d|^2 (\nabla d \cdot \nabla w)  [(\nabla d \cdot \nabla )\nabla d] \cdot \nabla w
\,dxdt.
\\
I_{8,7}
&=
-
\int_{Q_\delta}
16s^5\la^5\va^5a^3 |\nabla d|^2 (\nabla d \cdot \nabla w) \nabla(a|\nabla d|^2)\cdot \nabla w 
\,dxdt
\\&=
-2I_{8,8}
-
\int_{Q_\delta}
32s^5\la^5\va^5a^4 |\nabla d|^2 (\nabla d \cdot \nabla w) \left[(\nabla d \cdot \nabla)\nabla d \right] \cdot \nabla w 
\,dxdt.
\\
I_{9,3}
&=
-
\int_{Q_\delta}
24s^5\la^6\va^5 a^4 |\nabla d|^4 w \nabla d \cdot [(\nabla d \cdot \nabla)\nabla w]
\,dxdt
\geq
-C\low
-\frac12 I_{8,4}. 
\end{align*}
Moreover, we calculate
%
\begin{align*}
-I_{2,9}&=-\frac16 I_{9,2}
=
\int_{Q_\delta}
2 s^5\la^6\va^5 a^3 |\nabla d|^6 w A w 
\,dxdt
\\&\geq
-C\low
-
\int_{Q_\delta}
2 s^5\la^6\va^5 a^4 |\nabla d|^6 |\nabla w|^2
\,dxdt.
\\
-I_{4,9}
&=
\int_{Q_\delta}
2s^5\la^5\va^5 a^3 |\nabla d|^4 (\nabla a \cdot \nabla d) w \Delta w
\,dxdt
\\&\geq
-C\low
-
\int_{Q_\delta}
2s^5\la^5\va^5 a^3 |\nabla d|^4 (\nabla a \cdot \nabla d) |\nabla w|^2
\,dxdt.
\\
I_{7,9}
&= \frac12 I_{10,2} =
-
\int_{Q_\delta}
2s^5\la^5\va^5 a^2 |\nabla d|^4 (A d) w A w
\,dxdt
\\&\geq
-C\low
+
\int_{Q_\delta}
2s^5\la^5\va^5 a^3 |\nabla d|^4 (A d) |\nabla w|^2
\,dxdt. 
\\
I_{8,2}
&=
-
\int_{Q_\delta}
8s^5\la^5\va^5 a^3 |\nabla d|^4 (\nabla d \cdot \nabla w) A w
\,dxdt
\\&=
-\frac56 I_{8,4}
+3 I_{8,8}
-
\int_{Q_\delta}
20s^5\la^6\va^5 a^4 |\nabla d|^6 |\nabla w|^2
\,dxdt
\\&\quad
+
\int_{Q_\delta}
32 s^5\la^5\va^5 a^4 |\nabla d|^2  (\nabla d \cdot \nabla w) [(\nabla d \cdot \nabla) \nabla d] \cdot \nabla w
\,dxdt
\\&\quad
-
\int_{Q_\delta}
12 s^5\la^5\va^5 a^3 (\nabla a \cdot \nabla d) |\nabla d|^4  |\nabla w|^2
\,dxdt
\\&\quad
-
\int_{Q_\delta}
16 s^5\la^5\va^5 a^4 |\nabla d|^2 \left( [(\nabla d \cdot \nabla) \nabla d]\cdot \nabla d)\right) |\nabla w|^2
\,dxdt
\\&\quad
-
\int_{Q_\delta}
4s^5\la^5\va^5 a^3 |\nabla d|^4 (A d) |\nabla w|^2
\,dxdt
\\&\quad
+
\int_{Q_\delta}
8s^5\la^5\va^5 a^4 |\nabla d|^4 [(\nabla w \cdot \nabla) \nabla d] \cdot \nabla w
\,dxdt. 
\\
I_{1,9}
&=
-
\int_{Q_\delta}
4s^5\la^5\va^5 a^3 |\nabla d|^4 w \nabla d \cdot \nabla A w
\,dxdt
\\&\geq
-C\low
-10 I_{2,9}
-4 I_{4,9}
-2 I_{7,9}
-\frac12 I_{8,2}
\\&\quad
-
\int_{Q_\delta}
16 s^5\la^5\va^5 a^4 |\nabla d|^2 \left( [(\nabla d \cdot\nabla)\nabla d ] \cdot \nabla d\right) |\nabla w|^2
\,dxdt.
\end{align*}
Hence we have
\begin{align*}
I_{3,a}
&\geq
-C\low
+
\int_{Q_\delta}
40s^5\la^6\va^5a^4 |\nabla d|^4 |\nabla d \cdot \nabla w|^2
\,dxdt
\\&\quad
-
\int_{Q_\delta}
16 s^5\la^6\va^5 a^4 |\nabla d|^6 |\nabla w|^2
\,dxdt
\\&\quad
-
\int_{Q_\delta}
12 s^5\la^5\va^5 a^3  (\nabla a \cdot \nabla d)  |\nabla d|^4|\nabla w|^2
\,dxdt
\\&\quad
-
\int_{Q_\delta}
24 s^5\la^5\va^5 a^4 |\nabla d|^2 \left( [(\nabla d \cdot \nabla) \nabla d]\cdot \nabla d)\right) |\nabla w|^2
\,dxdt
\\&\quad
+
\int_{Q_\delta}
4s^5\la^5\va^5 a^4 |\nabla d|^4 [(\nabla w \cdot \nabla) \nabla d] \cdot \nabla w
\,dxdt. 
\end{align*}
Next we compute $I_{3,b}$. 
\begin{align*}
I_{6,9}
&=
\int_{Q_\delta}
4 s^5\la^5\va^5 a^4 |\nabla d|^4 w \left( \sum_{i,j=1}^n (\pp_i\pp_j  d) \pp_i \pp_j w \right)
\,dxdt
\\&\geq
-C\low
-
\int_{Q_\delta}
4 s^5\la^5\va^5 a^4 |\nabla d|^4 [(\nabla w \cdot \nabla)\nabla d]\cdot \nabla w
\,dxdt.
\\
I_{8,5}
&=
-
\int_{Q_\delta}
16s^5\la^5\va^5a^3 (A d) |\nabla d|^2 |\nabla d \cdot \nabla w|^2
\,dxdt
\geq
-C\low.
\\
I_{8,6}
&=
-
\int_{Q_\delta}
16s^5\la^5\va^5a^3 (\nabla a \cdot \nabla d) |\nabla d|^2 |\nabla d \cdot \nabla w|^2
\,dxdt
\geq
-C\low.
\\
I_{10,3}
&=
-
\int_{Q_\delta}
8s^5\la^5\va^5 a^3 (A d) |\nabla d|^2  w \nabla d \cdot [(\nabla d \cdot \nabla)\nabla w]
\,dxdt
\geq
-C\low.
\\
I_{11,2}
&=
-
\int_{Q_\delta}
12s^5\la^5\va^5 a^2 |\nabla d|^2 [\nabla d \cdot \nabla (a|\nabla d|^2)] w A w
\,dxdt
\\&\geq
-C\low
+
\int_{Q_\delta}
12s^5\la^5\va^5 a^3 (\nabla a \cdot \nabla d)|\nabla d|^4|\nabla w|^2
\,dxdt
\\&\quad
+
\int_{Q_\delta}
24s^5\la^5\va^5 a^4 |\nabla d|^2\left( [(\nabla d \cdot \nabla)\nabla d]\cdot \nabla d\right) |\nabla w|^2
\,dxdt.
\\
I_{11,3}
&=
-
\int_{Q_\delta}
24s^5\la^5\va^5 a^3[\nabla d \cdot \nabla (a|\nabla d|^2)]  w \nabla d \cdot [(\nabla d \cdot \nabla)\nabla w]
\,dxdt
\geq
-C\low.
\end{align*}
Summing up the above inequalities, we obtain
\begin{align*}
I_{3,b}
&\geq
-C\low
-
\int_{Q_\delta}
4 s^5\la^5\va^5 a^4 |\nabla d|^4 [(\nabla w \cdot \nabla)\nabla d]\cdot \nabla w
\,dxdt
\\&\quad
+
\int_{Q_\delta}
12s^5\la^5\va^5 a^3 (\nabla a \cdot \nabla d)|\nabla d|^4|\nabla w|^2
\,dxdt
\\&\quad
+
\int_{Q_\delta}
24s^5\la^5\va^5 a^4 |\nabla d|^2\left( [(\nabla d \cdot \nabla)\nabla d]\cdot \nabla d\right) |\nabla w|^2
\,dxdt.
\end{align*}
Thus we have the following inequality. 
\begin{align*}
I_3
&=I_{3,a}+I_{3,b}
\\&\geq
-C\low
+
\int_{Q_\delta}
40s^5\la^6\va^5a^4 |\nabla d|^4 |\nabla d \cdot \nabla w|^2
\,dxdt
\\&\quad
-
\int_{Q_\delta}
16 s^5\la^6\va^5 a^4 |\nabla d|^6 |\nabla w|^2
\,dxdt
.
\end{align*}

\noindent
\textbf{(\romannumeral4) Calculations of $I_{k,\ell}$ which give us 0th derivative.}

We calculate 9 terms including $I_{8,9}, \ldots, I_{11,9}$ and $I_{9,4}, \ldots, I_{9,8}$ which give us 
0th-order terms. We denote the sum of them by 
\begin{equation*}
I_4:=\sum_{k=8}^{11} I_{k,9}
+
\sum_{\ell=4}^{8} I_{9,\ell}.
\end{equation*}
On one hand, we have
\begin{align*}
I_{8,9}
&=
-
\int_{Q_\delta}
4s^7\la^7\va^7a^4|\nabla d|^6 w \nabla d \cdot \nabla w
\,dxdt
\\&=
\int_{Q_\delta}
14s^7\la^8\va^7a^4|\nabla d|^8  |w|^2
\,dxdt
+
\int_{Q_\delta}
2s^7\la^7\va^7a^3  (A d) |\nabla d|^6  |w|^2
\,dxdt
\\&\quad
+
\int_{Q_\delta}
6s^7\la^7\va^7  a^3 |\nabla d|^4 \left[\nabla d\cdot \nabla(a |\nabla d|^2)\right] |w|^2
\,dxdt.
\\
I_{9,9}
&=
-
\int_{Q_\delta}
6s^7\la^8\va^7a^4 |\nabla d|^8 |w|^2
\,dxdt. 
\\
I_{10,9}
&=
-
\int_{Q_\delta}
2s^7\la^7\va^7a^3 (A d) |\nabla d|^6 |w|^2
\,dxdt.
\\
I_{11,9}
&=
-
\int_{Q_\delta}
6s^7\la^7\va^7a^3 |\nabla d|^4 \left[ \nabla d \cdot \nabla (a|\nabla d|^2)\right] |w|^2
\,dxdt.
\end{align*}
On the other hand, we estimate
\begin{align*}
I_{9,4}
&=
-
\int_{Q_\delta}
72s^5\la^7\va^5 a^4 |\nabla d|^6 w \nabla d \cdot \nabla w
\,dxdt
\geq
-C\low.
\\
I_{9,5}
&=
-
\int_{Q_\delta}
24s^5\la^6\va^5 a^3 (A d) |\nabla d|^4 w \nabla d\cdot \nabla w 
\,dxdt
\geq
-C\low.
\\
I_{9,6}
&=
-
\int_{Q_\delta}
24s^5\la^6\va^5 a^3 (\nabla a \cdot \nabla d) |\nabla d|^4 w \nabla d\cdot \nabla w 
\,dxdt
\geq
-C\low.
\\
I_{9,7}
&=
-
\int_{Q_\delta}
24s^5\la^6\va^5 a^3 |\nabla d|^4 w \nabla (a|\nabla d|^2)\cdot \nabla w 
\,dxdt
\geq
-C\low.
\\
I_{9,8}
&=
\int_{Q_\delta}
12s^5\la^6\va^5 a^3 |\nabla d|^6 w \nabla a\cdot \nabla w 
\,dxdt
\geq
-C\low.
\end{align*}

Summing up all the above terms, we obtain 
\begin{equation*}
I_4
\geq
-C\low
+
\int_{Q_\delta}
8s^7\la^8\va^7a^4|\nabla d|^8  |w|^2
\,dxdt.
\end{equation*}

%
%
\noindent
\textbf{Step 4: Summing up the inequalities derived in Steps 1--3.}

Now we combine the inequalities that we have derived in the former steps and we have
\begin{align*}
&(P_1 w , P_2 w)_{L^2(Q_\delta)} 
\geq
-C\low
+I_0+I_1+I_2+I_3+I_4
\\&\geq
-C\low
+
\int_{Q_\delta}
8s\la^2\va a^2  |\nabla d \cdot \nabla A w|^2 
\,dxdt
\\&\quad
+
\int_{Q_\delta}
8s^3\la^4\va^3 a^2 |\nabla d|^4 |A w|^2
\,dxdt
-
\int_{Q_\delta}
16 s^3\la^4\va^3 a^4  |\nabla d|^2  |(\nabla d \cdot \nabla ) \nabla w|^2
\,dxdt
\\&\quad
+
\int_{Q_\delta}
40s^5\la^6\va^5a^4 |\nabla d|^4 |\nabla d \cdot \nabla w|^2
\,dxdt
-
\int_{Q_\delta}
16 s^5\la^6\va^5 a^4 |\nabla d|^6 |\nabla w|^2
\,dxdt
\\&\quad
+
\int_{Q_\delta}
8s^7\la^8\va^7a^4|\nabla d|^8  |w|^2
\,dxdt. 
\end{align*}
Together with \eqref{eq:prfC-1} and \eqref{eq:prfC-2}, this inequality yields
\begin{align}
\label{eq:prfC-3}
&\frac12
\| P_1 w\|_{L^2(Q_\delta)}^2
+
\int_{Q_\delta}
8s\la^2\va a^2  |\nabla d \cdot \nabla A w|^2 
\,dxdt
\\&
+
\int_{Q_\delta}
8s^3\la^4\va^3 a^2 |\nabla d|^4 |A w|^2
\,dxdt
\nonumber\\&
-
\int_{Q_\delta}
16 s^3\la^4\va^3 a^4  |\nabla d|^2  |(\nabla d \cdot \nabla ) \nabla w|^2
\,dxdt
\nonumber\\&
+
\int_{Q_\delta}
40s^5\la^6\va^5a^4 |\nabla d|^4 |\nabla d \cdot \nabla w|^2
\,dxdt
-
\int_{Q_\delta}
16 s^5\la^6\va^5 a^4 |\nabla d|^6 |\nabla w|^2
\,dxdt
\nonumber\\&
+
\int_{Q_\delta}
8s^7\la^8\va^7a^4|\nabla d|^8  |w|^2
\,dxdt
\nonumber\\&\leq 
\| Pw \|_{L^2(Q_\delta)}^2 
+C\low. 
\nonumber
\end{align}
In the next steps, we will estimate the non-positive terms on the left-hand side of \eqref{eq:prfC-3} from below. 

%
%
\noindent
\textbf{Step 5: Recovery of terms of $(\nabla d \cdot \nabla)\nabla w$.}

By integration by parts, we have
\begin{align*}
&
\int_{Q_\delta}
24s^3\la^4\va^3 a^4  |\nabla d|^2 |(\nabla d \cdot \nabla) \nabla w|^2
\,dxdt 
\\&\leq
C\low
-
\int_{Q_\delta}
24s^3\la^4\va^3 a^3  |\nabla d|^2  (\nabla d \cdot \nabla w) \nabla d \cdot \nabla A w
\,dxdt 
\\&\leq
C\low
+
\int_{Q_\delta}
6
s\la^2\va a^2  |\nabla d \cdot \nabla A w|^2
\,dxdt
+
\int_{Q_\delta}
24s^5\la^6\va^5 a^4  |\nabla d|^4  |\nabla d \cdot \nabla w|^2
\,dxdt 
,
\end{align*}
where 
we used 
\begin{align*}
&s^3\la^4\va^3 a^3  |\nabla d|^2  (\nabla d \cdot \nabla w) \nabla d \cdot \nabla A w\\
&=\left(\frac1{\sqrt{2}} s^{\frac12}\la\va^{\frac12} a \nabla d \cdot \nabla A w\right) 
	\left(\sqrt{2} s^{\frac52}\la^3\va^{\frac52}  a^2 |\nabla d|^2 \nabla d \cdot \nabla w \right) \\
&\leq
\frac14
s\la^2\va a^2  |\nabla d \cdot \nabla A w|^2
+
s^5\la^6\va^5 a^4  |\nabla d|^4  |\nabla d \cdot \nabla w|^2.
\end{align*}

Summing up this and \eqref{eq:prfC-3}, and absorbing similar integrals on the right-hand side into the left-hand side, we obtain
\begin{align}
\label{eq:prfC-4}
&\frac12
\| P_1 w\|_{L^2(Q_\delta)}^2
+
\int_{Q_\delta}
2 s\la^2\va a^2  |\nabla d \cdot \nabla A w|^2 
\,dxdt
\\&
+
\int_{Q_\delta}
8 s^3\la^4\va^3 a^2 |\nabla d|^4 |A w|^2
\,dxdt
\nonumber\\&\quad
+
\int_{Q_\delta}
8 s^3\la^4\va^3 a^4  |\nabla d|^2  |(\nabla d \cdot \nabla ) \nabla w|^2
\,dxdt
\nonumber\\&\quad
+
\int_{Q_\delta}
16 s^5\la^6\va^5a^4 |\nabla d|^4 |\nabla d \cdot \nabla w|^2
\,dxdt
\nonumber\\&\quad
-
\int_{Q_\delta}
16 s^5\la^6\va^5 a^4 |\nabla d|^6 |\nabla w|^2
\,dxdt
+
\int_{Q_\delta}
8 s^7\la^8\va^7a^4|\nabla d|^8  |w|^2
\,dxdt
\nonumber\\&\leq 
\| Pw \|_{L^2(Q_\delta)}^2 
+C\low.
\nonumber
\end{align}

%
%
\noindent
\textbf{Step 6: Rewriting the inequality in terms of $u$. }

We first calculate the term of $A u$. 
Since 
\begin{align*}
e^{s\psi} A u
&= 
 A w 
- 2s \la \va a \nabla d \cdot \nabla w
- s   \la   \va   (A d) w
- s   \la^2 \va   a |\nabla d|^2 w
+ s^2 \la^2 \va^2 a |\nabla d|^2 w 
, 
\end{align*}
we find that 
\begin{align}
\label{eq:prfC-5}
&\int_{Q_\delta}
s^3\la^4\va^3 a^2 |\nabla d|^4 |A u|^2 e^{2s\psi}
\,dxdt
\\&\leq
\int_{Q_\delta}
s^3\la^4\va^3 a^2 |\nabla d|^4 
\biggl(
\left(
A w 
+ s^2 \la^2 \va^2 a |\nabla d|^2 w
\right) 
- 2s \la \va a \nabla d \cdot \nabla w
\nonumber\\&\qquad
- \left(s \la   \va   (A d) w
+ s \la^2 \va   a |\nabla d|^2 w \right)
\biggr)^2
\,dxdt
\nonumber\\&\leq
\int_{Q_\delta}
8s^3\la^4\va^3 a^2 |\nabla d|^4 
\left(
A w 
+ s^2 \la^2 \va^2 a |\nabla d|^2 w
\right)^2
\,dxdt
\nonumber\\&\quad
+
\int_{Q_\delta}
8s^5\la^6\va^5a^4 |\nabla d|^4 |\nabla d \cdot \nabla w|^2
\,dxdt
+C\low
,
\nonumber
\end{align}
where we used 
\begin{align*}
(p+q+r)^2
&\leq
2(p+r)^2+2q^2
\leq
4p^2 +2q^2 +4r^2
\leq
8p^2 +2q^2 +4r^2.
\end{align*}

By integration by parts, 
we compute the first term on the right-hand side of the above inequality \eqref{eq:prfC-5}. 
\begin{align*}
&\int_{Q_\delta}
s^3\la^4\va^3 a^2 |\nabla d|^4 
\left(
A w 
+ s^2 \la^2 \va^2 a |\nabla d|^2 w
\right)^2
\,dxdt
\\&=
\int_{Q_\delta}
s^3\la^4\va^3 a^2 |\nabla d|^4 |A w|^2
\,dxdt
+
\int_{Q_\delta}
s^7\la^8\va^7a^4|\nabla d|^8  |w|^2
\,dxdt
\\&\quad
+
\int_{Q_\delta}
2s^5\la^6\va^5 a^3 |\nabla d|^6 
wA w 
\,dxdt
\\&\leq
C\low
+
\int_{Q_\delta}
s^3\la^4\va^3 a^2 |\nabla d|^4 |A w|^2
\,dxdt
+
\int_{Q_\delta}
s^7\la^8\va^7a^4|\nabla d|^8  |w|^2
\,dxdt
\\&\quad
-
\int_{Q_\delta}
2s^5\la^6\va^5 a^4 |\nabla d|^6 
|\nabla w|^2 
\,dxdt
.
\end{align*}
Thus from \eqref{eq:prfC-5}, we have
\begin{align*}
&\int_{Q_\delta}
s^3\la^4\va^3 a^2 |\nabla d|^4 |A u|^2 e^{2s\psi}
\,dxdt
\\&\leq
C\low
+
\int_{Q_\delta}
8s^3\la^4\va^3 a^2 |\nabla d|^4 |A w|^2
\,dxdt
\\&\quad
+
\int_{Q_\delta}
8s^5\la^6\va^5a^4 |\nabla d|^4 |\nabla d \cdot \nabla w|^2
\,dxdt
-
\int_{Q_\delta}
16s^5\la^6\va^5 a^4 |\nabla d|^6 
|\nabla w|^2 
\,dxdt
\\&\quad
+
\int_{Q_\delta}
8s^7\la^8\va^7a^4|\nabla d|^8  |w|^2
\,dxdt
.
\end{align*}
Combining this with \eqref{eq:prfC-4}, we obtain
\begin{align}
\label{eq:prfC-6}
&
\int_{Q_\delta}
s^3\la^4\va^3 a^2 |\nabla d|^4 |A u|^2 e^{2s\psi}
\,dxdt
\\&
+
\frac12
\| P_1 w\|_{L^2(Q_\delta)}^2
+
\int_{Q_\delta}
2 s\la^2\va a^2  |\nabla d \cdot \nabla A w|^2 
\,dxdt
\nonumber\\&
+
\int_{Q_\delta}
8 s^3\la^4\va^3 a^4  |\nabla d|^2  |(\nabla d \cdot \nabla ) \nabla w|^2
\,dxdt
\nonumber\\&\quad
+
\int_{Q_\delta}
8 s^5\la^6\va^5a^4 |\nabla d|^4 |\nabla d \cdot \nabla w|^2
\,dxdt
\nonumber\\&\leq 
\| Pw \|_{L^2(Q_\delta)}^2 
+C\low.
\nonumber
\end{align}
We mention that all the terms on the left-hand side of the above inequality are now non-negative. 

Next we estimate the terms of $\pp_t u$ and $A^2 u$. 
Since $\partial_t u = e^{-s\psi}(\partial_t w - s(\partial_t \psi)w)$, we have
\begin{align*}
&\int_{Q_\delta}
\frac1{s\va} |\pp_t u|^2 e^{2s\psi}
\,dxdt
\\&\leq
C
\int_{Q_\delta}
\frac1{s\va} |\pp_t w|^2
\,dxdt
+
C
\int_{Q_\delta}
s\va^{-1} |\pp_t \psi|^2 |w|^2
\,dxdt
\\&\leq
C
\int_{Q_\delta}
\frac1{s\va}
\left|
P_1 w 
-
4s\la\va a \nabla d \cdot \nabla A w 
-
4s^3\la^3\va^3 a^2 |\nabla d|^2 \nabla d\cdot \nabla w
\right|^2
\,dxdt
+
C\low
\\&\leq
C\| P_1 w\|_{L^2({Q_\delta})}^2
+
C
\int_{Q_\delta}
s\la^2\va a^2 |\nabla d \cdot \nabla A w |^2
\,dxdt
\\&\quad
+
C
\int_{Q_\delta}
s^5\la^6\va^5a^4 |\nabla d|^4 |\nabla d \cdot \nabla w|^2
\,dxdt
+C\low.
\end{align*}
Here we used the definition of $P_1w$ and the estimate 
$$
\int_{Q_\delta}
s\va^{-1} |\pp_t \psi|^2 |w|^2
\,dxdt
\le 
C\int_{Q_\delta}
s\va^{3} |w|^2
\,dxdt
\le
C\low.
$$
Thus by using \eqref{eq:prfC-6}, we derive
\begin{align}
\label{eq:prfC-7}
&
\int_{Q_\delta}
\frac1{s\va} |\pp_t u|^2 e^{2s\psi}
\,dxdt
\leq 
C
\| Pw \|_{L^2({Q_\delta})}^2 
+C\low.
\end{align}
Noting that $A^2 u = \pp_t u - e^{-s\psi}Pw$, we have
\begin{equation*}
\int_{Q_\delta}
\frac1{s\va} |A^2 u|^2 e^{2s\psi}
\,dxdt
\leq
C
\| Pw \|_{L^2({Q_\delta})}^2 
+
C
\int_{Q_\delta}
\frac1{s\va} |\pp_t u|^2 e^{2s\psi}
\,dxdt.
\end{equation*}
We combine this with \eqref{eq:prfC-6} and \eqref{eq:prfC-7} and we obtain
\begin{align}
\label{eq:prfC-8}
&
\int_{Q_\delta}
\frac1{s\va} \left( |\pp_t u|^2 + |A^2 u|^2 \right)e^{2s\psi} 
\,dxdt
+
\int_{Q_\delta}
s^3\la^4\va^3 a^2 |\nabla d|^4 |A u|^2 e^{2s\psi}
\,dxdt
\\
&
+
\int_{Q_\delta}
s\la^2\va a^2  |\nabla d \cdot \nabla A w|^2 
\,dxdt
+
\int_{Q_\delta}
s^3\la^4\va^3 a^4  |\nabla d|^2  |(\nabla d \cdot \nabla ) \nabla w|^2
\,dxdt
\nonumber\\&\quad
+
\int_{Q_\delta}
s^5\la^6\va^5a^4 |\nabla d|^4 |\nabla d \cdot \nabla w|^2
\,dxdt
\nonumber\\
&\leq 
C
\| Pw \|_{L^2({Q_\delta})}^2 
+C\low.
\nonumber
\end{align}

By using the positivity of $a$, 
$|\nabla d(x)|>\sigma$, $x\in\overline{\Omega}$, 
the left-hand side of \eqref{eq:prfC-8} can be estimated from below and we have
\begin{align}
\label{eq:prfC-9}
&
\int_{Q_\delta}
\frac1{s\va} \left( |\pp_t u|^2 + |A^2 u|^2 \right)e^{2s\psi} 
\,dxdt
+
\int_{Q_\delta}
s^3\la^4\va^3 |A u|^2 e^{2s\psi}
\,dxdt
\\
&
+
\int_{Q_\delta}
s\la^2\va |\nabla d \cdot \nabla A w|^2 
\,dxdt
+
\int_{Q_\delta}
s^3\la^4\va^3 |(\nabla d \cdot \nabla ) \nabla w|^2
\,dxdt
\nonumber\\&\quad
+
\int_{Q_\delta}
s^5\la^6\va^5 |\nabla d \cdot \nabla w|^2
\,dxdt
\nonumber\\
&\leq 
C
\| Pw \|_{L^2({Q_\delta})}^2 
+C\low.
\nonumber
\end{align}

To obtain the lower-order derivatives of $u$, 
we use the Carleman estimate for the elliptic operator $A$. 
By Lemma \ref{lem:ce_elliptic2}, the inequality \eqref{eq:prfC-9} yields
\begin{align}
\label{eq:prfC-10}
&
\int_{Q_\delta}
		\Biggl[
		\frac1{s\va} \left( |\pp_t u|^2 +|A^2 u|^2\right) 
		+ s^3\la^4\va^3  |A u|^2
		\\&\qquad
		+ s^2\la^4\va^2  \sum_{i,j=1}^n |\pp_i \pp_j u |^2 
		+ s^4\la^6\va^4 |\nabla  u|^2 
		+ s^6\la^8\va^6 |u|^2
		\Biggr] e^{2s\psi}
\,dxdt
\nonumber\\
&
+
\int_{Q_\delta}
s\la^2\va |\nabla d \cdot \nabla A w|^2 
\,dxdt
+
\int_{Q_\delta}
s^3\la^4\va^3 |(\nabla d \cdot \nabla ) \nabla w|^2
\,dxdt
\nonumber\\&\quad
+
\int_{Q_\delta}
s^5\la^6\va^5 |\nabla d \cdot \nabla w|^2
\,dxdt
\nonumber\\
&\leq 
C
\| Pw \|_{L^2({Q_\delta})}^2 
+C\low.
\nonumber
\end{align}
Hereafter we assume that $s\ge \max\{1,\delta^2\}$ and $\la\ge 1$ are large enough such that the above inequality \eqref{eq:prfC-10} holds true. 

Finally we recover the third derivative $\nabla A u$. We carry on integration by parts and obtain
%
\begin{align*}
&\int_{Q_\delta}
s\la^2\va |\nabla A u|^2 e^{2s\psi}
\,dxdt
\\&\leq
C\low
-
\int_{Q_\delta}
s\la^2\va  a^{-1} (A^2 u) (A u) e^{2s\psi}
\,dxdt
+
\int_{Q_\delta}
2s^3\la^4\va^3 |\nabla d|^2 |A u|^2 e^{2s\psi}
\,dxdt
\\&\leq
C\low
+
C
\int_{Q_\delta}
\frac1{s\va} |A^2 u|^2 e^{2s\psi}
\,dxdt
+
C
\int_{Q_\delta}
s^3\la^4\va^3 |A u|^2 e^{2s\psi}
\,dxdt.
\end{align*}
Putting this together with \eqref{eq:prfC-10} leads to
\begin{align}
\nonumber
&
\int_{Q_\delta}
		\Biggl(
		\frac1{s\va} \left( |\pp_t u|^2 +|A^2 u|^2\right) 
		+s\la^2\va |\nabla A u|^2 e^{2s\psi}
		+ s^3\la^4\va^3  |A u|^2
		\\&\qquad
		+ s^2\la^4\va^2  \sum_{i,j=1}^n |\pp_i \pp_j u |^2 
		+ s^4\la^6\va^4 |\nabla  u|^2 
		+ s^6\la^8\va^6 |u|^2
		\Biggr) e^{2s\psi}
\,dxdt
\nonumber\\
&
+
\int_{Q_\delta}
s\la^2\va |\nabla d \cdot \nabla A w|^2 
\,dxdt
+
\int_{Q_\delta}
s^3\la^4\va^3 |(\nabla d \cdot \nabla ) \nabla w|^2
\,dxdt
\nonumber\\&\quad
+
\int_{Q_\delta}
s^5\la^6\va^5 |\nabla d \cdot \nabla w|^2
\,dxdt
\nonumber\\
&\leq 
C
\| Pw \|_{L^2({Q_\delta})}^2 
+C\low
\nonumber
.
\end{align}
In the end, 
taking sufficiently large $\la$ and $s$, 
we may absorb the lower-order term $C\low$ on the right-hand side of the above inequality 
into the left-hand side, 
which concludes the proof of Theorem \ref{thm:ce0}. 
\end{proof}
%
%
\section*{Appendix}
%
%
\subsection*{Carleman Estimates for Elliptic Equations}
We state the Carleman estimate for the following elliptic operator: 
\begin{equation*}
\widetilde{A}v(x) = 
\dd (\widetilde{a}\nabla v(x) ) 
-\widetilde{\mathbf{b}}(x)\cdot \nabla v(x)
-\widetilde{c}(x) v(x), \quad x \in \Omega, 
\end{equation*}
where
$\widetilde{a}\in C^1(\overline{\Omega})$, 
$\widetilde{\mathbf{b}}\in \{L^\infty(\overline{\Omega})\}^n$,
$\widetilde{c}\in L^\infty(\overline{\Omega})$,  and 
$\frac1{\widetilde{\mu}}<\widetilde{a}(x)<\widetilde{\mu}$, $x\in \overline\Omega$ with  a positive constant $\widetilde{\mu}$. 
We define 
\begin{equation*}
\widetilde{\va}(x) := \theta_0e^{\la d(x)}, \quad 
\widetilde{\psi}(x) := \theta_0 \left(e^{\la d(x)}-e^{2\la \|d\|_{C(\overline{\Omega})}}\right), \quad x \in \Omega.
\end{equation*}
for $\theta_0>0$, where $d$ is the same function introduced in Section {\bf 2}.

\begin{lem}
\label{lem:ce_elliptic1}
There exists $\la_0>0$ such that for any $\la\ge \la_0$, 
we may choose $s_0(\la)>0$ satisfying the following: 
there exists a constant $C=C(s_0,\la_0)>0$ such that 
\begin{equation*}
\int_{\Omega} 
		\Biggl(
		\frac1{s\widetilde{\va}} \left( \sum_{i,j}^n |\pp_i\pp_j v|^2\right) 
		+
		s\la^2\widetilde{\va} |\nabla v|^2
		+ 
		s^3\la^4\widetilde{\va}^3  |v|^2
		\Biggr) e^{2s\widetilde{\psi}}
\,dx\leq
C\int_{\Omega} \left|\widetilde{A}v\right|^2 e^{2s\widetilde{\psi}} \,dx
\end{equation*}
for all $s\ge s_0$ and
all $v \in H^2(\Omega)$ that is compactly supported in $\Omega$. 
\end{lem}
We may prove the above lemma by the integration by parts 
via an argument similar to that used in the proof of Carleman estimate for elliptic equations. Hence we refer to, for example, \cite{Yamamoto09} and omit the details of the proof here. 

Using Lemma \ref{lem:ce_elliptic1} with suitable $\theta_0= \theta_0(t)$ and integrating over $I_\delta$, we can show the following lemma. 
\begin{lem}
\label{lem:ce_elliptic2}
There exists $\la_0>0$ such that for any $\la\ge \la_0$, 
we may choose $s_0(\la)>0$ satisfying the following: 
there exists a constant $C=C(s_0,\la_0)>0$ such that 
\begin{align*}
&
\int_{Q_\delta} 
		\Biggl(
		s^2\la^4\va^2  \sum_{i,j=1}^n |\pp_i \pp_j u |^2 
		+ s^4\la^6\va^4 |\nabla  u|^2 
		+ s^6\la^8\va^6 |u|^2
		\Biggr)  e^{2s\psi}
\,dxdt 
\\&\leq
C\int_{Q_\delta} s^3\la^4\va^3 \left|Au\right|^2  e^{2s\psi} \,dxdt
\end{align*}
for all $s\ge s_0$ and
all $u \in L^2\left(0,T;H^2(\Omega)\right)$ with $\supp\, u \subset \Omega \times \overline{I_\delta}$. 
\end{lem}
%
%
\subsection*{Carleman Estimate for a Third-order Partial Differential Equations}
By modifying Lemma 4.9 in \cite{Kawamoto-Machida20}, 
we may show the following lemma.
\begin{lem}
\label{lem:ce_3rdpde}
Let $\mathbf{p}\in C^1(\overline{\Omega})$. 
We assume that $|\mathbf{p}(x)\cdot \nabla d(x)|\geq m_0$, 
$x\in \overline{\Omega_0}$ with a positive constant $m_0$. 
There there exists $\la_0>0$ such that for any $\la\ge \la_0$, 
we may choose $s_0(\la)>0$ satisfying: 
there exists a constant $C=C(s_0,\la)>0$ such that
\begin{align*}
&
\int_\Omega
\Biggl(
s\widetilde{\va}
\sum_{i,j,k=1}^n |\pp_i \pp_j \pp_k v|^2
+
(s\widetilde{\va})^{2}
|\nabla \Delta v|^2
\\
&\qquad
+
(s\widetilde{\va})^{3}
\sum_{i,j=1}^n |\pp_i \pp_j v|^2
+(s\widetilde{\va})^{5}
\left(|\nabla v|^2+ |v|^2 \right)
\Biggr)
e^{2s \widetilde{\psi}}\,dx
\\
&
\leq
C\int_\Omega \left(  
|\nabla (\mathbf{p}  \cdot \nabla \Delta v)|^2 
+ 
|\mathbf{p} \cdot \nabla \Delta v|^2\right) e^{2s \widetilde{\psi}}\,dx,
\end{align*}
for all $s\ge s_0$ and all $v \in H^4(\Omega)$ satisfying 
$v(x)=0$, $x\in D\cup\pp\Omega$. 
\end{lem}

%
%

\end{document}